\documentclass[a4paper,10pt]{amsart}
\usepackage[english]{babel}

\usepackage{amsmath,amssymb,amsthm}
\usepackage{mathrsfs}
\usepackage[alphabetic, abbrev]{amsrefs}
\usepackage{enumerate}
\usepackage[bookmarksopen,final]{hyperref}
\usepackage[all]{xy}
\newdir{ >}{{}*!/-5pt/@{>}} 
\SelectTips{cm}{} 

\numberwithin{equation}{section} 
\usepackage[utf8]{inputenc}

\newtheorem{theorem}[equation]{Theorem}
\newtheorem{thm}[equation]{Theorem}
\newtheorem{cor}[equation]{Corollary}
\newtheorem{lemma}[equation]{Lemma}
\newtheorem{prop}[equation]{Proposition}

\theoremstyle{definition}
\newtheorem{defn}[equation]{Definition}
\newtheorem{rk}[equation]{Remark}
\newtheorem{question}[equation]{Question}
\newtheorem{eg}[equation]{Example}
\newtheorem{con}[equation]{Construction}

\newcommand{\cC}{\mathcal{C}}
\newcommand{\cI}{\mathcal{I}}
\newcommand{\cM}{\mathcal{M}}
\newcommand{\cN}{\mathcal{N}}

\DeclareMathOperator{\hocolim}{hocolim}
\DeclareMathOperator{\colim}{colim}
\DeclareMathOperator{\const}{const}
\DeclareMathOperator{\Hom}{Hom}
\DeclareMathOperator{\supp}{supp}

\newcommand{\xra}{\xrightarrow}
\newcommand{\xla}{\xleftarrow}

\newcommand{\id}{{\mathrm{id}}}

\newcommand{\sm}{\wedge}
\newcommand{\iso}{\cong}
\newcommand{\concat}{\sqcup}
\newcommand{\bld}[1]{{\mathbf{#1}}}
\newcommand{\sset}{\mathrm{sSet}}
\newcommand{\set}{\mathrm{Set}}
\newcommand{\tm}{\mathrm{tame}}

\newcommand{\doilink}[1]{\href{http://dx.doi.org/#1}{doi:#1}}

\title{Homotopy invariance of convolution products}

\author{Steffen Sagave}
\address{IMAPP, Radboud University Nijmegen, The Netherlands}
\email{s.sagave@math.ru.nl}

\author{Stefan Schwede}
\address{Mathematisches Institut, Universit\"at Bonn, Germany}
\email{schwede@math.uni-bonn.de}

\begin{document}

\begin{abstract}
  The purpose of this paper is to show that various convolution
  products are fully homotopical, meaning that they preserve weak
  equivalences in both variables without any cofibrancy hypothesis.
  We establish this property for diagrams of simplicial sets indexed
  by the category of finite sets and injections and for tame
  $M$-simplicial sets, with $M$ the monoid of injective self-maps of
  the positive natural numbers.  We also show that a certain
  convolution product studied by Nikolaus and the first author is
  fully homotopical. This implies that every presentably symmetric
  monoidal $\infty$-category can be represented by a symmetric
  monoidal model category with a fully homotopical monoidal product.
\end{abstract}
\date{\today}
\maketitle

\section{Introduction}
The convolution product of functors between symmetric monoidal
categories was introduced by the category theorist Brian Day
\cite{day:closed}.  It made a prominent appearance in homotopy theory
when Jeff Smith and Manos Lydakis simultaneously and independently
introduced the smash products of symmetric spectra \cite{HSS} and the
smash product of $\Gamma$-spaces \cite{lydakis:Gamma}, respectively.
Since then, convolution products have become an essential ingredient
in the homotopy theory toolkit, and many more examples were introduced
and studied in the context of stable homotopy theory
\cite{lydakis:simplicial, mmss}, unstable homotopy theory
\cite{lind:diagram, Sagave-S_diagram, Sagave-S_group-compl},
equivariant homotopy theory \cite{dundas-rondigs-ostvaer:enriched,
  mandell-may}, motivic homotopy theory \cite{dundas-rondigs-ostvaer:motivic,
  jardine:motivic}, and $\infty$-category theory~\cite{glasman},
to name just a few.

In order for convolution products to be homotopically meaningful,
they must `interact nicely' with some relevant notion of weak
equivalence (except in the $\infty$-categorical case,
where this is an intrinsic feature). In typical situations,
homotopically useful convolution products come with compatible closed model
structures. The compatibility then includes the `pushout product
property' which implies that the convolution product is homotopy
invariant when all objects involved are cofibrant.  This cofibrancy
requirement for the homotopy invariance often leads to cofibrancy
hypotheses in applications. The proof of the homotopy
invariance for cofibrant objects typically proceeds by a cellular
reduction argument to free or representable objects.

Given the history of the subject, one would not expect convolution
products to be \emph{fully homotopical}, i.e., to preserve the
relevant weak equivalences in both variables without any cofibrancy
hypothesis. Nevertheless, there are already two non-obvious instances
of full homotopy invariance. On the one hand, the box product of
orthogonal spaces (also known as $\mathscr I$-functors,
$\mathscr I$-spaces or $\mathcal I$-spaces) is fully
homotopical~\cite[Theorem\,1.3.2]{schwede:global}. On the other hand,
the operadic product of $\mathcal L$-spaces, i.e., spaces with a
continuous action of the topological monoid $\mathcal L$ of linear
self-isometries of $\mathbb R^\infty$, is fully
homotopical~\cite[Theorem\,1.21]{schwede:orbispaces}. In both cases,
the weak equivalences can be chosen to be the global equivalences, and
the proofs make essential use of explicit homotopies that are not
available for more discrete or combinatorial index categories.

The purpose of the present paper is to show that there are more
interesting instances of fully homotopical convolution products.

\subsection*{\texorpdfstring{$\cI$}{I}-spaces}

Let $\cI$ be the category with objects the finite sets $\bld{m}=\{1,\dots, m\}$
and with morphisms the injections. An \emph{$\cI$-space} is a covariant
functor from $\cI$ to simplicial sets. We say that a map of
$\cI$-spaces $f\colon X \to Y$ is an \emph{$\cI$-equivalence} if the induced
map of homotopy colimits $f_{h\cI}\colon X_{h\cI} \to Y_{h\cI}$ is a
weak equivalence. The category of $\cI$-spaces $\sset^{\cI}$ has a Day
type convolution product $\boxtimes$, with $X\boxtimes Y$ defined as
the left Kan extension of the object-wise cartesian product
$(\bld{m},\bld{n}) \mapsto X(\bld{m})\times Y(\bld{n})$ along the
concatenation $- \concat - \colon \cI \times \cI \to \cI$. The
interest in $\cI$-spaces comes from the fact that homotopy types of
$E_{\infty}$ spaces can be represented by commutative $\cI$-space
monoids, i.e., by commutative monoids with respect to $\boxtimes$, and
that many interesting $E_{\infty}$ spaces arise from explicit
commutative $\cI$-space monoids (see e.g.~\cite[\S 1.1]{Sagave-S_diagram}). 

Our first main result states that $\boxtimes$ is fully homotopical. 
 \begin{theorem}\label{thm:intro:boxtimes-fully-homotopical}
  Let $X$ be an $\cI$-space. Then the functor $X \boxtimes - $
  preserves $\cI$-equivalences.
\end{theorem}
This can be viewed as a discrete analog
of~\cite[Theorem\,1.3.2]{schwede:global}. The homotopy invariance of
the $\boxtimes$-product of $\cI$-spaces was previously established by
Schlichtkrull and the first author under the additional hypothesis
that one of the factors is flat cofibrant, see
\cite[Proposition\,8.2]{Sagave-S_diagram}. With
Theorem~\ref{thm:intro:boxtimes-fully-homotopical}, several flatness
hypotheses required in~\cite{Sagave-S_group-compl} turn out to be
unnecessary.  For example, Proposition 2.20, Proposition 2.23, Lemma
2.25, Proposition 2.27, Corollary 2.29, Proposition 4.2, Corollary 4.3
and Proposition 4.8 of \cite{Sagave-S_group-compl} hold without the
flatness hypotheses, and Theorem 1.2 and Lemma 4.12 of
\cite{Sagave-S_group-compl} hold without the cofibrancy hypothesis on
the commutative $\cI$-space monoid.

\subsection*{Tame \texorpdfstring{$M$}{M}-spaces}
Let $M$ be the monoid of injective self-maps of the set of
positive natural numbers. A \emph{tame $M$-set} is a set with an
$M$-action that satisfies a certain local finiteness condition (see
Definition~\ref{def:finitely supprted}). Actions of $M$ were studied by the second author
in \cite{schwede:homotopy} where it was shown that the homotopy groups
of symmetric spectra admit tame $M$-actions that for example detect
semistability.

A \emph{tame $M$-space} is a simplicial set with an $M$-action that is
tame in every simplicial degree. The resulting category
$\sset^{M}_{\tm}$ also has a convolution product $\boxtimes$ that is
analogous to the operadic product of $\mathcal L$-spaces. We say that
a map $f\colon X \to Y$ of tame $M$-spaces is an
\emph{$M$-equivalence} if it induces a weak equivalence
$f_{hM} \colon X_{hM} \to Y_{hM}$ of homotopy colimits (which are
given by bar constructions).

Our second main result states that $\boxtimes$ is fully homotopical. 

\begin{theorem}\label{thm:intro:boxtimes-M-fully-homotopical}
  Let $X$ be a tame $M$-space. Then the functor $X \boxtimes - $
  preserves $M$-equivalences between tame $M$-spaces.
\end{theorem}

This theorem can be viewed as a discrete analog
of~\cite[Theorem\,1.21]{schwede:orbispaces} and we use it to prove
Theorem~\ref{thm:intro:boxtimes-fully-homotopical}. It is also
interesting in itself: we show that the category of tame $M$-spaces is
equivalent to the reflective subcategory of the category of $\cI$-spaces
given by the flat $\cI$-spaces. This implies that tame $M$-spaces and
$M$-equivalences provide a model for the homotopy theory of spaces,
and that $E_{\infty}$ spaces can be represented by strictly
commutative monoids in $(\sset^{M}_{\tm} ,\boxtimes)$.
We also show that commutative $\boxtimes$-monoids of tame $M$-spaces
can be described as algebras over an injection operad, see Theorem \ref{thm:Iset and com}.

\subsection*{Presentably symmetric monoidal
  \texorpdfstring{$\infty$}{infinity}-categories} We now leave the
framework of specific multiplicative models 
for the homotopy theory of
spaces studied so far and consider general homotopy theories with a
symmetric monoidal product. One convenient way to encode these are the
\emph{symmetric monoidal $\infty$-categories} introduced by
Lurie~\cite{Lurie_HA}.  More precisely, we shall consider
\emph{presentably symmetric monoidal $\infty$-categories}. By
definition, these are presentable $\infty$-categories with a symmetric
monoidal structure that preserves colimits separately in each
variable.

When passing to underlying $\infty$-categories, combinatorial
symmetric monoidal model categories give rise to presentably symmetric
monoidal $\infty$-categories~\cite[Example\,4.1.3.6 and Proposition\,4.1.3.10]{Lurie_HA}.
Nikolaus and the first author showed
in~\cite{Nikolaus-S_presentably} that conversely, every presentably
symmetric monoidal $\infty$-category $\cC$ is represented by a
combinatorial symmetric monoidal model category $\cM$.  By
construction, $\cM$ is a left Bousfield localization of a certain
\emph{contravariant $\cI$-model structure} on an over category
$\sset^{\cI}/N$ with $N$ a commutative $\cI$-simplicial set. Here both
$N$ and the localization depend on $\cC$. The symmetric monoidal
product of $\cM$ is induced by the $\boxtimes$-product on
$\sset^{\cI}$ and the multiplication of $N$ and can therefore also be
viewed as a convolution product. Combining some of the results about
the interaction of $\cI$-spaces and tame $M$-spaces used to prove
Theorems~\ref{thm:intro:boxtimes-fully-homotopical}
and~\ref{thm:intro:boxtimes-M-fully-homotopical} with an analysis of
the contravariant $\cI$-model structure, we show that the convolution
product on $\cM$ is fully homotopical and thus provide the following
stronger variant of~\cite[Theorem\,1.1]{Nikolaus-S_presentably}.
\begin{theorem}\label{thm:presentably-sym-monoidal-intro}
  Every presentably symmetric monoidal $\infty$-category is
  represented by a simplicial, combinatorial and left proper symmetric
  monoidal model category with a fully homotopical monoidal product.
\end{theorem}
The authors view this theorem as a surprising result since point-set
level models for multiplicative homotopy theories tend to be not fully
homotopical (or to be not known to have this property). It leads to
symmetric monoidal model categories which one may not have expected to
exist. For example, applying the theorem to the $\infty$-category of
spectra provides a symmetric monoidal model for the stable homotopy
category with a fully homotopical smash product. To the authors'
knowledge, none of the common symmetric monoidal models for the stable
homotopy category such as~\cite{ekmm, HSS, lydakis:simplicial, mmss}
is known to have this property (see also the questions below).  As
another example, Theorem~\ref{thm:presentably-sym-monoidal-intro}
implies that the homotopy theory underlying the derived category of a
commutative ring $k$ admits a fully homotopical model. The existence
of non-trivial $\mathrm{Tor}$-terms shows that in general, this fails
badly for the usual tensor product of chain complexes of
$k$-modules. As a last example, applying
Theorem~\ref{thm:presentably-sym-monoidal-intro} to the presentably
symmetric monoidal $\infty$-category of $\infty$-operads constructed
by Lurie~\cite[Chapter\,2.2.5]{Lurie_HA} leads to a fully homotopical
model for the category of topological operads with the Boardman--Vogt
tensor product.

The model category provided by~\cite[Theorem\,1.1]{Nikolaus-S_presentably}
also has the desirable feature that it
lifts to operad algebras over simplicial operads and that weak
equivalences of operads induce Quillen equivalences between categories
of operad algebras~\cite[Theorem\,2.5]{Nikolaus-S_presentably}. This
implies that like for positive model structures on diagram spectra or
diagram spaces, $E_{\infty}$ objects can be represented by strictly
commutative ones. Together with the full homotopy invariance, this
model thus provides a setup where homotopical algebra is particularly
simple.

\subsection*{Open questions}
The above discussion and the close connection
between $\cI$-spaces and symmetric spectra
lead to the following question.

\begin{question}
Is the smash product of symmetric spectra of simplicial sets
fully homotopical for the stable equivalences?
\end{question}

At the time of this writing, and to the best of the authors'
knowledge, this question is open, and the authors would like to see
this being sorted out. Analogously, one can consider symmetric spectra
of topological spaces, or the orthogonal case: 

\begin{question}
Is the smash product of orthogonal spectra 
fully homotopical for the stable equivalences?
\end{question}

Affirmative answers to these questions would make numerous cofibrancy
assumptions in applications of symmetric or orthogonal spectra
superfluous and thereby lead to substantial simplifications. 

\subsection*{Organization}
Section~\ref{sec:tame-M-sets} provides combinatorial results about tame $M$-sets
that are used in the rest of the paper.
In Section~\ref{sec:M-spaces} we study the homotopy theory of tame $M$-spaces
and prove Theorem~\ref{thm:intro:boxtimes-M-fully-homotopical}.
Section~\ref{sec:I-spaces} is about homotopical properties
of the $\boxtimes$-product of $\cI$-spaces
and contains the proof of Theorem~\ref{thm:intro:boxtimes-fully-homotopical}.
In Section~\ref{sec:flat-I} we identify tame $M$-spaces with flat $\cI$-spaces
and construct model structures on these categories.
Section~\ref{sec:presentably-symmetric-monoidal} discusses contravariant model structures
and provides the proof of Theorem~\ref{thm:presentably-sym-monoidal-intro}.
In Appendix~\ref{app:algebras-inj-operad} we identify commutative $\boxtimes$-monoids
in tame $M$-sets with  tame algebras over the injection operad,
and we supply an alternative characterization
of the $\boxtimes$-product of tame $M$-spaces as an operadic product. 

\subsection*{Conventions}
In this paper we follow a common abuse of language in that the term
\emph{$\cI$-space} means a functor from the injection category $\cI$ to the
category of \emph{simplicial sets} (as opposed to a functor to some
category of topological spaces). Similarly, an \emph{$M$-space} is a
simplicial set with an action of the injection monoid $M$.  The use of
simplicial sets is essential for several of our arguments, and we can
offer no new insight about the convolution products of $\cI$-topological
spaces and $M$-topological spaces.

\subsection*{Acknowledgments} This project was initiated by a question
asked by Christian Schlicht\-krull during the Problem Session of the
workshop ``Combinatorial Categories in Algebra and Topology'' held at
the Universit\"at Osnabr\"uck in November 2018. We would like to thank
the workshop organizers Tim R\"omer, Oliver R\"ondigs and Markus
Spitzweck for bringing us together at this occasion, and we would like
to thank Christian Schlicht\-krull for useful discussions.
The second author is a member of the Hausdorff Center for Mathematics,
funded by the Deutsche Forschungsgemeinschaft (DFG, German Research Foundation)
under Germany's Excellence Strategy (GZ 2047/1, project ID 390685813).

\section{The structure of tame \texorpdfstring{$M$}{M}-sets}\label{sec:tame-M-sets}
In this section we discuss tame $M$-sets, i.e., sets equipped with an
action of the injection monoid that satisfy a certain local finiteness
condition.  The main result is Theorem \ref{thm:complement}, which
says that every tame $M$-set $W$ decomposes as a disjoint union of
$M$-subsets arising in a specific way from certain $\Sigma_m$-sets
associated with $W$.  The arguments of this section are combinatorial
in nature, but they crucially enter into the homotopical analysis in the subsequent sections. 

\begin{defn}
  The {\em injection monoid} $M$ is the monoid of injective self-maps
  of the set $\omega=\{1,2,3,\dots\}$ of positive natural numbers with
  monoid structure given by composition. 
  An $M$-set is a set with a left $M$-action.
\end{defn}

\begin{defn}\label{def:finitely supprted}
  An element $x$ of an $M$-set $W$ is {\em supported} on a
  subset $A$ of $\omega$ if the following condition holds: for every
  injection $f\in M$ that fixes $A$ elementwise, the relation
  $ f x= x$ holds.
  An $M$-set $W$ is {\em tame} if every element is supported
  on some finite subset of $\omega$.
\end{defn}

We write $\set^M$ for the category of $M$-sets and $M$-equivariant maps,
and we denote by $\set^M_{\tm}$ the full subcategory of tame $M$-sets.

Clearly, if $x$ is supported on $A$ and $A\subseteq B\subseteq \omega$,
then $x$ is supported on $B$. Every element is supported on all of~$\omega$.
An element is supported on the empty set if and only if it is fixed by $M$.

\begin{prop}\label{prop:intersect supports}
Let $x$ be an element of an $M$-set $W$.
If $x$ is supported on two finite subsets $A$ and $B$ 
of $\omega$, then $x$ is supported on the intersection $A\cap B$.
\end{prop}
\begin{proof}
We let $f\in M$ be an injection that fixes $A\cap B$ elementwise.
We let $m$ be the maximum of the finite set $A\cup B\cup f(A)$ 
and define $\sigma\in M$ as the involution that interchanges 
$j$ with $j+m$ for all $j\in B-A$, i.e.,
\[ \sigma(j)\ = \
\begin{cases}
  j+m & \text{\ for $j\in B-A$, }\\
  j-m & \text{\ for $j\in (B-A)+m$, and}\\
  j  & \text{\ for $j\not\in (B-A)\cup ((B-A)+m)$.}
\end{cases}\]
In particular, the map $\sigma$ fixes the set $A$ elementwise.
Since $A$ and $f(A)$ are both disjoint from $B+m$, 
we can choose a bijection $\gamma\in M$ such that
\[ \gamma(j)\ = \
\begin{cases}
  f(j) & \text{\ if $j\in A$, and}\\
\ j  & \text{\ for $j\in B+m$.}
\end{cases}\]
Then $f$ can be written as the composition
\[ f \ = \  \sigma(\sigma \gamma\sigma)(\sigma \gamma^{-1} f)\ .\]
In this decomposition the factors $\sigma$ and $\sigma \gamma^{-1} f$
fix $A$ pointwise, and the factor $\sigma \gamma\sigma$
fixes $B$ pointwise. So 
\[\sigma x\ = \ (\sigma \gamma\sigma)x \ = \ (\sigma \gamma^{-1} f) x
\ = \ x \]
because $x$ is supported on $A$ and on $B$.
This gives
\[ f x \ = \ \sigma (\sigma \gamma \sigma) (\sigma \gamma^{-1} f) x  \ =\  x\  .  \]
 Since $f$ was any injection fixing $A\cap B$ elementwise,
the element $x$ is supported on $A\cap B$.
\end{proof}

\begin{defn}\label{def:support} 
Let $x$ be an element of an $M$-set.
The {\em support} of $x$ is the intersection of all finite subsets of $\omega$
on which $x$ is supported.  
\end{defn}

We write $\supp(x)$ for the support of $x$ and agree that
$\supp(x)=\omega$ if $x$ is not finitely supported. Proposition
\ref{prop:intersect supports} then shows that $x$ is supported on its
support $\supp(x)$.  It is important that in
Definition~\ref{def:support} the intersection is only over {\em
  finite} supporting subsets. Indeed every element is supported on the
set $\omega-\{j\}$ for every $j\in\omega$ (because the only injection
that fixes $\omega-\{j\}$ elementwise is the identity). So without the
finiteness condition the intersection in Definition~\ref{def:support}
would always be empty.

\begin{prop}\label{prop:finite support} 
Let $W$ be an $M$-set and $x\in W$.
  \begin{enumerate}[\em (i)]
  \item If the injections $f,g\in M$ agree on $\supp(x)$, 
    then $f x=g x$.
  \item For every injection $f\in M$, the relation 
    \[ \supp(f x)\ \subseteq\ f(\supp(x)) \]
    holds. If $x$ is finitely supported, then
    $\supp(f x)=f(\supp(x))$, and $f x$ is also finitely supported.
  \end{enumerate}
\end{prop}
\begin{proof}
(i) If $x$ is not finitely supported, then $\supp(x)=\omega$, so $f=g$ and
there is nothing to show. If $x$ is finitely supported, we can choose a bijection $h\in M$ 
that agrees with $f$ and $g$ on $\supp(x)$. 
Then $h^{-1}f$ and $h^{-1}g$ fix the support of $x$ elementwise and hence 
$(h^{-1}f) x=x=(h^{-1}g) x$. Thus
\[ f x\ =\ h((h^{-1}f) x)\ = \ h x\ =\ h((h^{-1}g) x)\ = \ g x  \ .\]

(ii) 
We let $g\in M$ be an injection that fixes $f(\supp(x))$ elementwise.
Then $g f$ agrees with $f$ on $\supp(x)$, so
\[ g(f x) \ = \ (g f) x \ = \ f x \]
by part (i). We have thus shown that $f x$ is supported on $f(\supp(x))$.

If $x$ is finitely supported, then  $f(\supp(x))$ is finite, and 
this proves that $f x$ is finitely supported and 
$\supp(f x)\subseteq f(\supp(x))$. For the reverse inclusion we choose $h\in M$
such that $h f$ fixes $\supp(x)$ elementwise; then $(h f) x=x$. 
Applying the argument to $h$ and $f x$ (instead of $f$ and $x$) gives
\[ \supp(x)\ = \ \supp((h f) x)\ = \
 \supp(h (f x))\ \subseteq\ h(\supp(f x)) \ ,\]
and thus
\[ f(\supp(x))\ \subseteq\ f(h(\supp(f x))) \  = \ 
(f h)(\supp(f x)) \  = \ \supp(f x)\ . \]
The last equation uses that $f h$ is the identity on the set $f(\supp(x))$,
hence also the identity on the subset $\supp(f x)$. 
This proves the desired relation when $x$ is finitely supported.
\end{proof}

As for any monoid, the category of $M$-sets is complete and cocomplete,
and limits and colimits are created on underlying sets.
Proposition \ref{prop:finite support} (ii) shows that for every $M$-set $W$,
the subset
\[  W_\tau \ = \ \{x\in W \ |\  \text{$x$ has finite support}\}\]
is closed under the $M$-action, and hence a tame $M$-subset of $W$.
A morphism $u\colon V\to W$ of $M$-sets preserves supports in the sense of the containment relation
\[ \supp(u x)\ \subseteq \ \supp(x) \]
for all $x\in V$. In particular, $M$-linear maps send finitely
supported elements to finitely supported elements.  So if the
$M$-action on $V$ is tame, then every $M$-linear map $f\colon V\to W$ has
image in $W_\tau$. This shows part (i) of the following lemma whose
other parts are then formal consequences.

\begin{lemma}\phantomsection\label{lem:pushouts-seq-colimits-M-sets}
  \begin{enumerate}[\em (i)]
  \item The functor 
    \[ (-)_\tau \colon \set^M\ \to \ \set^M_{\tm}\ , \quad W\
    \longmapsto \ W_\tau\]
    is right adjoint to the inclusion, with the inclusion
    $W_\tau\to W$ being the counit of the adjunction.
  \item The category of tame $M$-sets is cocomplete, and the forgetful
    functor to sets preserves colimits.
  \item The category of tame $M$-sets is complete, and limits can be
    calculated by applying the functor $(-)_\tau$ to limits in the
    category of $M$-sets.\qed
  \end{enumerate}
\end{lemma}

Our next aim is to prove that every tame $M$-set is a disjoint union of $M$-sets
of the form $\cI_m\times_{\Sigma_m}A_m$ for varying $m\geq 0$ and $\Sigma_m$-sets $A_m$.

\begin{prop}\label{prop:algebraic consequences}
  Let $W$ be a tame $M$-set and $f\in M$. Then the action of $f$ is an
  injective map $W\to W$, and its image consists precisely of those
  elements that are supported on the set $f(\omega)$.
\end{prop}
\begin{proof} For injectivity we consider any $x,y\in W$ with
  $f x=f y$.  Since $f$ is injective and $x$ and $y$ are finitely
  supported, we can choose $h\in M$ such that $h f$ is the identity on
  the support of $x$ and the support of $y$.  Then $x=h f x=h f y=y$.

  It remains to identify the image of the action of $f$. For all $x\in W$ we have 
  \[ \supp(f x)\ \subseteq \ f(\supp(x))\ \subseteq \ f(\omega) \ ,\]
  so $f x$ is supported on $f(\omega)$.
  Now suppose that $z\in W$ is supported on $f(\omega)$.
  Then $f$ restricts to a bijection from $f^{-1}(\supp(z))$ to $\supp(z)$.
  We choose a bijection $g\in M$ such that $f g$ is the identity on $\supp(z)$.
  Then $f g z=z$, and hence $z$ is in the image of the action of $f$.
\end{proof}

Proposition \ref{prop:algebraic consequences} can fail for non-tame $M$-sets.
An example is given by letting $f\in M$ act on the set
$\{0,1\}$ as the identity if the image of $f$ has finite complement,
and setting $f(0)=f(1)=0$ if its image has infinite complement.

\begin{defn}
  Given a set $A$, we write $\cI_A$ for the $M$-set of injective maps from $A$ to $\omega$,
  where the monoid $M$ acts by postcomposition. 
\end{defn}

\begin{eg} \label{ex:I_m} 
  The $M$-set $\cI_{\emptyset}$ has only one element, so the $M$-action is
  necessarily trivial. If $A$ is finite and non-empty, then $\cI_A$ is countably
  infinite and the $M$-action is non-trivial, but tame: the support of
  an injection $\alpha\colon A\to\omega$ is its image $\alpha(A)$.

  If $A$ is a finite subset of $\omega$, then the $M$-set $\cI_A$ represents the functor
  of taking elements with support in $A$:
  the inclusion $\iota_A\colon A\to\omega$ is supported on $A$, and for every $M$-set $W$,
  the evaluation map
  \[ \Hom_M(\cI_A, W) \ \xra{\ \iso\ }\ \{x\in W\colon \supp(x)\subseteq A\}\ ,
    \quad \varphi \ \longmapsto \ \varphi(\iota_A) \]
  is bijective.
\end{eg}
  
 When $A=\mathbf m=\{1,\dots,m\}$, we write $\cI_m=\cI_{\mathbf m}$.
 The $M$-set $\cI_m$ comes with a commuting right action of the symmetric group $\Sigma_m$
  by precomposition. So we can -- and will -- view 
  $\cI_m$ as a $\Sigma_m$-object in the category of tame $M$-sets.
  If $K$ is any left $\Sigma_m$-set, we can form the tame $M$-set
  $\cI_m\times_{\Sigma_m} K$
  by coequalizing the two $\Sigma_m$-actions on the product
  $\cI_m\times K$. This construction yields a functor
  \[ \cI_m\times_{\Sigma_m} - \colon \set^{\Sigma_m} \ \to \ \set^M_{\tm}\ .\] 

Now we prove that every tame $M$-set is a disjoint union of $M$-sets
that arise from the functors $\cI_m\times_{\Sigma_m}-$ for varying $m\geq 0$.
For an $M$-set $W$ and $m\geq 0$, we write
\[  s_m(W)\ = \ \{ x\in W \ : \ \supp(x)=\bld{m}\}\]
for the subset of elements whose support is the set $\bld{m}=\{1,\dots,m\}$.
This is a $\Sigma_m$-invariant subset of $W$ by Proposition \ref{prop:finite support} (ii).
For example, $s_0(W)$ is the set of $M$-fixed elements. 
We alert the reader that
the $\Sigma_m$-sets $s_m(W)$ are {\em not} functors in $W$:
a morphism $f\colon V\to W$ of $M$-sets may decrease the support,
and hence it need not take $s_m(V)$ to $s_m(W)$.

We define a morphism of $M$-sets
\begin{equation} \label{eq:define_psi}
  \psi_m^W\colon \cI_m\times_{\Sigma_m}s_m(W)\ \to\  W
  \text{\qquad by\qquad}  [\alpha,x]\ \longmapsto \ \tilde\alpha x\ ,
\end{equation}
where $\tilde\alpha\in M$ is any injection that extends
$\alpha\colon \mathbf m\to\omega$.  This assignment is well-defined by
Proposition \ref{prop:finite support} (i).

\begin{theorem}\label{thm:complement}
  For every tame $M$-set $W$, the morphisms $\psi^W_m$ assemble into an isomorphism
  of $M$-sets
  \[ \coprod_{m\geq 0} \cI_m\times_{\Sigma_m} s_m(W)\ \xra{\ \iso \ } \ W\ .\]
\end{theorem}
\begin{proof}
  We write
  \[  C_m(W)\ = \ \{ x\in W \ : \ |\supp(x)|=m\}\]
  for the subset of elements whose support has cardinality $m$.
  This is an $M$-invariant subset by Proposition \ref{prop:finite support} (ii).
  Since the $M$-action on $W$ is tame, $W$ is the disjoint union of
  the $M$-subsets $C_m(W)$ for $m\geq 0$.
  The image of $\psi^W_m$ is contained in $C_m(W)$, so it remains to show that the morphism
  \[ \psi_m^W\colon \cI_m\times_{\Sigma_m}s_m(W)\ \to\  C_m(W)\]
  is an isomorphism for every $m\geq 0$.
   
  For surjectivity we consider any element $x\in C_m(W)$.
  Since the support of $x$ has cardinality $m$, we can choose a bijection $h\in M$ such that
  $h(\supp(x))= \{1,\dots,m\}$. Then $\supp(h x)=h(\supp(x))=\{1,\dots,m\}$
  by Proposition \ref{prop:finite support} (ii), so $h x$ belongs to $s_m(W)$.
  Moreover,
  \[ \psi_m^W[h^{-1}|_{\{1,\dots,m\}}, h x] \ = \ h^{-1}(h x)\ = \ x\ ,\]
  so the map $\psi_m^W$ is surjective.
  
  For injectivity we consider $f,g\in \cI_m$ and $x,y\in s_m(W)$
  such that $\psi_m^W[f,x]=\psi_m^W[g, y]$.  Then
  \[ \supp(\psi_m^W[f,x]) \ = \ f(\supp(x))\ = \  f(\{1,\dots,m\})\ ,\]
  and similarly $\supp(\psi_m^W[g,y])= g(\{1,\dots,m\})$.
  So the class $\psi_m^W[f,x]=\psi_m^W[g,y]$ is supported on the
  intersection of $f(\{1,\dots,m\})$ and $g(\{1,\dots,m\})$.
  Since the support of $\psi_m^W[f,x]$ has cardinality $m$,
  we conclude that $f(\{1,\dots,m\})=g(\{1,\dots,m\})$.  There is thus a
  unique element $\sigma\in\Sigma_m$ such that $g=f\sigma$.  We extend
  $\sigma$ to an element $\tilde\sigma\in M$ by fixing all numbers
  greater than $m$.

  We choose a bijection $h\in M$ that extends $f$; then
  $h \tilde\sigma$ extends $f\sigma=g$.  Hence
  \[ x \ = \ h^{-1}(h x)\ = \ h^{-1}\psi_m^W[f,x]\ =\ h^{-1}\psi_m^W[g, y] \ = \ h^{-1} h \tilde \sigma y\
    \ = \ \tilde \sigma y \ . \]
  So we conclude that
  \[ [f,x] \ = \ [f,\tilde\sigma y] \ = \ [f\sigma,y]\ = \ [g,y]\ . \]
  This proves that the map $\psi_m^W$ is injective, hence bijective.
\end{proof}

We now define the box product of $M$-sets.

\begin{defn}\label{def:box-product}
  Let $X$ and $Y$ be $M$-sets. The {\em box product} $X\boxtimes  Y$
  is the subset of the product consisting of those pairs $(x,y)\in X\times Y$
  such that $\supp(x)\cap\supp(y)=\emptyset$.
\end{defn}

As we explain in Proposition \ref{prop:convolution versus operadic},
the box product of $M$-sets is closely related to an {\em operadic product}.

\begin{prop}\label{prop:describe operadic}
  Let $X$ and $Y$ be $M$-sets. 
  \begin{enumerate}[\em (i)]
  \item 
    The box product $X\boxtimes Y$ is an $M$-invariant subset of $X\times Y$.
  \item If $X$ and $Y$ are tame, then so is $X\boxtimes  Y$.
\end{enumerate}
\end{prop}
\begin{proof}
  (i) For all $(x,y)\in X\times Y$ and all $f\in M$ we have
  \[ \supp(f x)\cap \supp(f y) \ \subseteq \ f(\supp( x))\cap f(\supp(y)) \
    = \ f(\supp( x)\cap \supp( y)) \ .\]
  So the pair $f(x,y)=(f x,f y)$ belongs to $X\boxtimes  Y$
  whenever $(x,y)$ does.
  Part (ii) follows from the relation
  $\supp_{X\times Y}(x,y) = \supp_X(x)\cup \supp_Y(y)$.
\end{proof}

The associativity, symmetry and unit isomorphisms of the cartesian product
of $M$-sets clearly restrict to the box product;
for example, the associativity isomorphism is given by
\[ (X\boxtimes Y)\boxtimes Z \ \xrightarrow{\iso} \
  X\boxtimes(Y\boxtimes Z) \ ,\quad ((x,y),z)\ \longmapsto\ (x,(y,z))\ .
\]
Hence they inherit the coherence conditions required for a symmetric monoidal product.
We thus conclude:

\begin{prop}
  The box product is a symmetric monoidal structure on the category of $M$-sets
  with respect to the associativity, symmetry and unit isomorphisms inherited
  from the cartesian product.
  The box product restricts to a symmetric monoidal structure on the category of
  tame $M$-sets.
  Every one-element $M$-set is a unit object for the box product.\qed
\end{prop}

\begin{eg}\label{eg:box of semifree}
  The tame $M$-set $\cI_m$ of injective maps from $\mathbf m=\{1,\dots,m\}$
  to $\omega$ was discussed in Example \ref{ex:I_m}.
  We define a morphism of $M$-sets
  \[ \cI_{m+n}\ \to \ \cI_m\times \cI_n \ , \quad f\ \longmapsto\ (f i^1,f i^2)\ ,\]
  where $i^1\colon \mathbf m\to\mathbf{m+n}$ is the inclusion and
  $i^2\colon \mathbf n\to\mathbf{m+n}$ sends $i$ to $m+i$.
  This morphism is injective and its image consists precisely of those pairs $(\alpha,\beta)$
  such that 
  \[  \supp(\alpha)\cap\supp(\beta)\ = \ \alpha(\mathbf m)\cap\beta(\mathbf n)\ =\ \emptyset\ .\]
  So the map restricts to an isomorphism of $M$-sets 
  \[ \rho \colon \cI_{m+n}\ \xrightarrow{\iso} \ \cI_m\boxtimes  \cI_n \ .\]
\end{eg}

\begin{lemma}\label{lem:unnatural-box-vs-product}
  Let $X$ and $Y$ be tame $M$-sets. There is a bijection of sets
  $X\boxtimes Y \to X \times Y$ that is natural for $M$-equivariant maps in $Y$.
\end{lemma}
\begin{proof}
  We write $Y_{[A]}$ for the subset of elements of $Y$ that are
  supported on $\omega-A$.  By definition, the underlying 
  set of $X\boxtimes Y$ is the disjoint union, over $x\in X$, of the sets
  \[ \{x\}\times Y_{[\supp(x)]} \ ,\]
  and this decomposition is natural for $M$-equivariant morphisms in $Y$.
  We let $f\in M$ be any injection with $f(\omega)=\omega-\supp(x)$.
  Then the action of $f$ is a bijection from $Y$ onto the subset $Y_{[\supp(x)]}$,
  by Proposition \ref{prop:algebraic consequences}. 
  So the underlying set of $X\boxtimes  Y$ bijects with the underlying set of $X\times Y$.
\end{proof}

We warn the reader that the bijection between $X\boxtimes Y$ and $X\times Y$
constructed in the previous lemma is in general neither $M$-equivariant nor natural in $X$.

\begin{cor}\label{cor:boxtimes-preserves-colimits}
If $X$ is a tame $M$-set, then $X\boxtimes - \colon \set^M_{\tm} \to  \set^M_{\tm}$ preserves colimits. \qed
\end{cor}

\section{Homotopy theory of \texorpdfstring{$M$}{M}-spaces}\label{sec:M-spaces}
In this section we introduce two notions of equivalence for $M$-spaces,
the {\em weak equivalences} and the {\em $M$-equivalences},
and we show that the box product of $M$-spaces is fully homotopical for
both kinds of equivalences.

\begin{defn}\begin{enumerate}[(i)]
  \item An {\em $M$-space} is a simplicial set equipped with a left
    action of the injection monoid $M$. We write $\sset^M$ for the
    category of $M$-spaces.
  \item An $M$-space is {\em tame} if for every simplicial dimension
    $k$, the $M$-set of $k$-simplices is tame. We write
    $\sset^M_{\tm}$ for the full subcategory of $\sset^M$
    whose objects are the tame $M$-spaces.
  \end{enumerate}
\end{defn}

\begin{defn}
  A morphism $f\colon X\to Y$ of $M$-spaces is
  \begin{enumerate}[(i)]
  \item a {\em weak equivalence} if it is a weak equivalence of
    underlying simplicial sets after forgetting the $M$-action, and
    \item an {\em $M$-equivalence} if the induced
    morphism on homotopy colimits (bar constructions)
    $f_{h M}\colon X_{h M}\to Y_{h M}$ is a weak equivalence of simplicial sets.
  \end{enumerate}
\end{defn}
The weak equivalences of $M$-spaces are analogous
to the $\pi_*$-isomorphisms of symmetric spectra, or the $\cN$-equivalences of $\cI$-spaces.
The $M$-equivalences are the more important class of equivalences since
the homotopy theory represented by tame $M$-spaces relative to
$M$-equivalences is the homotopy theory of spaces (see
Corollary~\ref{cor:model-on-tame-M-spaces} below). Dropping the
tameness assumption, a result by
Dugger~\cite[Theorem\,5.2]{Dugger_replacing} implies that the
$M$-equivalences are the weak equivalence in a model structure on
$\sset^M$ that is Quillen equivalent to spaces. This uses that the
classifying space of $M$ is weakly
contractible~\cite[Lemma\,5.2]{schwede:homotopy}.

Every weak equivalence of $M$-spaces induces a weak equivalence on
homotopy $M$-orbits; hence weak equivalences are in particular
$M$-equivalences.  The converse is not true: for every injective map
$\alpha\colon \bld{m} \to \bld{n}$, the induced restriction morphism
\[ \alpha^* \colon \cI_n \ \to \ \cI_m\ , \quad f \ \longmapsto \ f\circ \alpha \]
is an $M$-equivalence because the $M$-homotopy orbits of source and
target are weakly contractible, by the following Example \ref{eg:hM of I_m}.
However, for $m\ne n$, the morphism is not a weak equivalence.

\begin{eg}\label{eg:hM of I_m}
  The $M$-set $\cI_m$ of injective maps $\mathbf m\to\omega$ was
  discussed in Example \ref{ex:I_m}.  We claim that the simplicial set
  $(\cI_m)_{h M}$ is weakly contractible.  To this end we observe that
  $(\cI_m)_{h M}$ is the nerve of the translation category $T_m$ whose
  object set is $\cI_m$, and where morphisms from $f$ to $g$ are all
  $h\in M$ such that $h f=g$. Given any two $f,g\in \cI_m$, we can
  choose a bijection $h\in M$ such that $h f=g$, and then $h$ is an
  isomorphism from $f$ to $g$ in $T_m$.  Since all objects of the
  category $T_m$ are isomorphic, its nerve is weakly equivalent to the
  classifying space of the endomorphism monoid of any of its
  objects. All these endomorphism monoids are isomorphic to the
  injection monoid $M$, so we conclude that the simplicial set
  $(\cI_m)_{h M}\iso N(T_m)$ is weakly equivalent to the classifying
  space of the injection monoid $M$.  The classifying space of $M$ is
  weakly contractible by \cite[Lemma\,5.2]{schwede:homotopy}, hence so
  is the simplicial set $(\cI_m)_{h M}$.
\end{eg}

If $X$ and $Y$ are $M$-spaces, their box product
$X \boxtimes  Y$ is given by the levelwise box product of
$M$-sets, introduced in Definition \ref{def:box-product}.
The next results show that the box product with any tame $M$-space
preserves weak equivalences and $M$-equivalences.

\begin{thm}\label{thm:box preserves weak M}
  For every tame $M$-space $X$, the functor $X\boxtimes  -$ preserves weak equivalences
  between tame $M$-spaces.
\end{thm}
\begin{proof}
  We start with the special case where $X$ is a tame $M$-{\em set} (as
  opposed to an $M$-space). For an $M$-space $Y$,
  Lemma~\ref{lem:unnatural-box-vs-product} implies that $X\boxtimes Y$
  is isomorphic to the underlying simplicial set of $X\times Y$ with the isomorphism being natural in $Y$. 
  So $X\boxtimes  -$ preserves weak equivalences
  between tame $M$-spaces because $X \times -$ does. 

  Now we treat the general case. The box product is formed dimensionwise in the
  simplicial direction, and realization (i.e., diagonal) of bisimplicial sets 
  preserves weak equivalences. So the general case follows from the special case of $M$-sets.
\end{proof}

Since $X\boxtimes  Y$ was defined as an $M$-invariant subspace of the product $X\times Y$,
the two projections restrict to morphisms of $M$-spaces
\[ p^1 \colon X\boxtimes  Y\ \to \ X \text{\qquad and\qquad} p^2\colon X\boxtimes  Y\ \to \ Y \ .\]

\begin{thm}\label{thm:invariance M}
  For all tame $M$-spaces $X$ and $Y$, the morphism
  \[ (p^1_{h M}, p^2_{h M}) \colon (X\boxtimes  Y)_{h M} \ \to \ X_{h M}\times Y_{h M}\]
  is a weak equivalence of simplicial sets.
\end{thm}
\begin{proof}
  We start with the special case where $X=\cI_m/H$ and $Y=\cI_n/K$ for some $m,n\geq 0$,
  for some subgroup $H$ of $\Sigma_m$, and some subgroup $K$ of $\Sigma_n$.
  By Example~\ref{eg:box of semifree}, the box product $\cI_m\boxtimes  \cI_n$ is isomorphic to $\cI_{m+n}$.
  Hence by Example~\ref{eg:hM of I_m}, source and target of the $(H\times K)$-equivariant morphism
  \[ (p^1_{h M}, p^2_{h M}) \colon (\cI_m\boxtimes  \cI_n)_{h M} \ \to \ (\cI_m)_{h M}\times (\cI_n)_{h M}\]
  are weakly contractible.
  The actions of $H$ on $(\cI_m)_{h M}$ and of $K$ on $(\cI_n)_{h M}$ are free.
  Similarly, the  isomorphism
  $(\cI_m\boxtimes  \cI_n)_{h M}\iso (\cI_{m+n})_{h M}$ provided by Example \ref{eg:box of semifree}
  shows that the action of $H\times K$ on $(\cI_m\boxtimes  \cI_n)_{h M}$ is free.
  Any equivariant map between free $(H\times K)$-spaces that is a weak equivalence
  of underlying simplicial sets induces a weak equivalence on orbit spaces.
  So the morphism
  \[ ( p^1_{h M}, p^2_{h M})/(H\times K) \colon (\cI_m\boxtimes
  \cI_n)_{h M}/ (H\times K) \ \to \ (\cI_m)_{h M}/H\times (\cI_n)_{h
    M}/K\]
  is a weak equivalence of simplicial sets.  The claim now follows
  because the box product preserves colimits in each variable by
  Corollary~\ref{cor:boxtimes-preserves-colimits}, and because the
  canonical morphism
  \[  ( (\cI_m)_{h M})/H \ \to \ (\cI_m/H)_{h M} \]
  is an isomorphism of simplicial sets.
  
  Now we assume that $X$ and $Y$ are $M$-sets. Then $X$ is isomorphic
  to a disjoint union of $M$-sets of the form $\cI_m\times_{\Sigma_m}A_m$, for varying
  $m\geq 0$ and $\Sigma_m$-sets $A_m$, by Theorem \ref{thm:complement}.
  The $\Sigma_m$-set $A$, in turn, is isomorphic to a disjoint union of $\Sigma_m$-sets
  of the form $\Sigma_m/H$, for varying subgroups $H$ of $\Sigma_m$.
  The box product of $M$-spaces and the cartesian product of simplicial sets
  distribute over disjoint unions, and homotopy orbits preserve disjoint unions.
  So source and target of the morphism under consideration decompose as disjoint unions
  of $M$-sets considered in the previous paragraph.
  Weak equivalences of simplicial sets are stable under disjoint unions, so
  this reduces the case of $M$-sets to the case of the previous paragraph.

  Now we treat the general case. All relevant construction commute with realization
  of simplicial objects in $M$-spaces (i.e., diagonal of simplicial sets).
  So the general case follows from the special case of $M$-sets because
  realization of bisimplicial sets is homotopical.  
\end{proof}

This result implies Theorem~\ref{thm:intro:boxtimes-M-fully-homotopical} from the introduction:

\begin{proof}[Proof of Theorem~\ref{thm:intro:boxtimes-M-fully-homotopical}]
  Since the product of simplicial sets preserves weak equivalences in both variables, it
  follows from Theorem~\ref{thm:invariance M} that the $\boxtimes $-product is fully homotopical
  with respect to the $M$-equivalences. 
\end{proof}

\begin{rk}
  The cartesian product of $M$-spaces, with diagonal $M$-action,
  is {\em not} homotopical for $M$-equivalences.
  For example, the unique morphism $\cI_1\to\ast$ to the terminal $M$-space
  is an $M$-equivalence by Example \ref{eg:hM of I_m}.
  However, the product $\cI_1\times \cI_1$ is the disjoint union
  of $\cI_1\boxtimes  \cI_1$ (which is isomorphic to $\cI_2$)
  and the diagonal (which is isomorphic to $\cI_1$).
  So $(\cI_1\times \cI_1)_{h M}$ consists of two weakly contractible components,
  and the projection to the first factor
  \[ \cI_1\times \cI_1 \ \to \ \cI_1 \]
  is {\em not} an $M$-equivalence.

  We call a tame $M$-space $X$ \emph{semistable}
  if the canonical map $X \to X_{hM}$ is a weak equivalence of simplicial sets.
  This notion is analogous to semistability of symmetric spectra
  (see~\cite[\S 5.6]{HSS} and~\cite[\S 4]{schwede:homotopy}) and
  $\cI$-spaces (see~\cite[\S 2.5]{Sagave-S_group-compl}).  For semistable
  tame $M$-spaces $X$ and $Y$, one can show that the inclusion
  $X\boxtimes Y\to X\times Y$ is both a weak equivalence and an
  $M$-equivalence. Hence the product of semistable tame $M$-spaces
  is fully homotopical for $M$-equivalences.
\end{rk}

\begin{eg}[Homotopy infinite symmetric product]
  We review an interesting class of tame $M$-spaces due to
  Schlichtkrull \cite{schlichtkrull:infinite_symmetric}.
  Schlichtkrull's paper is written for topological spaces, but we work with simplicial sets instead.
  We let $X$ be a based simplicial set, with basepoint denoted by $\ast$.
  We define $X^\infty$ as the simplicial subset of
  $X^\omega$ consisting, in each simplicial dimension, of those functions $\alpha\colon \omega\to X$
  such that $\alpha(i)=\ast$ for almost all $i\in\omega$.
  The injection monoid $M$ acts on $X^\infty$ by
  \[ (f\alpha)(j) \ = \
    \begin{cases}
      \alpha(i) & \text{\ if $f(i)=j$, and}\\
     \ \ast & \text{\ if $j\not\in f(\omega)$.}
    \end{cases}
\]
So informally speaking, one can think of $f \alpha$ as the
composite $\alpha\circ f^{-1}$, with the caveat that $f$ need not be
invertible.  An element $\alpha\in X^\infty$ is supported on the
finite set $\omega-\alpha^{-1}(\ast)$, so the $M$-action on $X^\infty$
is tame.  Schlichtkrull's main result about this construction is that
for connected $X$, the $M$-homotopy orbit space $X^\infty_{h M}$ is
weakly equivalent to $Q(X)$, the underlying infinite loop space of the
suspension spectrum of $X$, see
\cite[Theorem\,1.2]{schlichtkrull:infinite_symmetric}.
In other words, $X^\infty$ is a model for $Q(X)$ in the world of tame $M$-spaces.

The $M$-space $X^\infty$ has additional structure: it is a commutative monoid
for the box product of $M$-spaces. Indeed, a multiplication
\[ \mu_X \colon X^\infty\boxtimes  X^\infty \ \to \ X^\infty \]
is given by
\[ \mu_X(\alpha,\beta)(i) \ = \   \
  \begin{cases}
    \alpha(i) & \text{\ if $\alpha(i)\ne\ast$, and}\\
    \beta(i) & \text{\ if $\beta(i)\ne\ast$, and}\\
    \ \ast & \text{\ if $\alpha(i)=\beta(i)=\ast$.}
  \end{cases}
\]
This assignment makes sense because whenever $\alpha$ and $\beta$ have disjoint support,
then for all $i\in\omega$, at least one of the elements $\alpha(i)$ and $\beta(i)$
must be the basepoint of $X$. The multiplication map $\mu_X$ is clearly associative and commutative,
and the constant map to the basepoint of $X$ is a neutral element.
With this additional structure, $X^\infty$ has the following universal property.
We let $M$ act on $\omega_+\sm X$ by $f(i,x)=(f(i),x)$.
Then $X^\infty$ is the free commutative $\boxtimes$-monoid,
in the category of based tame $M$-spaces, generated by the tame $M$-space $\omega_+\sm X$.
This universal property, or direct inspection, shows that the functor
\[ \sset_*\ \to \ \sset^M_{\tm} \ , \quad X \ \longmapsto \ X^\infty\]
from based simplicial sets takes coproducts to box products (which are coproducts in the category
of commutative $\boxtimes $-monoids). More precisely, for all based simplicial sets $X$ and $Y$,
the composite map
\[ X^\infty\boxtimes  Y^\infty \ \xra{i_X^\infty\boxtimes 
  i_Y^\infty}\ (X\vee Y)^\infty\boxtimes  (X\vee Y)^\infty \
\xra{\mu_{X\vee Y}} (X\vee Y)^\infty\]
is an isomorphism of tame $M$-spaces, where $i_X\colon X\to X\vee Y$
and $i_Y\colon Y\to X\vee Y$ are the inclusions. The analogous
statement for $\cI$-spaces is~\cite[Lemma 3.1]{schlichtkrull:infinite_symmetric}.
\end{eg}

\section{The box product of \texorpdfstring{$\mathcal I$}{I}-spaces}\label{sec:I-spaces}
In this section we prove that the box product of $\cI$-spaces is fully
homotopical for the classes of $\cI$-equivalences and
$\mathcal N$-equivalences.
Our strategy is to reduce each of these claims to the corresponding
one for the box product of tame $M$-spaces.  The key tool is a certain
functor from $\cI$-spaces to tame $M$-spaces that matches the box
products and the relevant notions of equivalences.

We let $\cI$ denote the category whose objects are the sets
$\mathbf m=\{1,\dots,m\}$ for $m\geq 0$, with $\mathbf 0$ being the
empty set. Morphisms in $\cI$ are all injective maps.  The category
$\cI$ is thus a skeleton of the category of finite sets and
injections. (Other authors denote this category by $\mathbb I$ or $F I$,
and some use the letter $\cI$ for the orthogonal counterpart.)

\begin{defn}
  An {\em $\cI$-space} is a functor from the category $\cI$ to the
  category of simplicial sets.  A morphism of $\cI$-spaces is a
  natural transformation of functors, and we write $\sset^{\cI}$ for
  the category of $\cI$-spaces. 
\end{defn}
We denote by $\cN$ the non-full subcategory of $\cI$ containing all
objects, but only the inclusions as morphisms. In other words, $\cN$
is the category associated with the partially ordered set
$(\mathbb N,\leq )$.
\begin{con}[From $\cI$-spaces to tame $M$-spaces]\label{con:M action}
  We let $X$ be an $\cI$-space. We write
  \[ X(\omega)\ = \ \colim_\cN X \]
  for the colimit of $X$ formed over the non-full subcategory $\cN$.
  We observe that the simplicial set $X(\omega)$ supports a natural
  action by the injection monoid $M$, defined as follows.  We let
  $[x]\in X(\omega)$ be represented by $x\in X(\mathbf m)$, and we let
  $f\in M$ be an injection. We set $n=\max(f(\mathbf m))$ and write
  $\tilde f\colon \mathbf m\to\mathbf n$ for the restriction of
  $f$. Then we get a well-defined $M$-action by setting 
  \[ f\, [x]\ = \ [X(\tilde f)(x)]\ . \]
  Moreover, $x$ is supported on $\bld{m}$, and so this $M$-action is tame.
\end{con}

In Construction~\ref{con:tame-M-to-I} below we will exhibit this construction
as a left adjoint functor $(-)(\omega) \colon \set^{\cI} \to \set^M_{\tm}$.
To analyze its homotopical properties, we use the following notions. 

\begin{defn}
  A morphism $f\colon X\to Y$ of $\cI$-spaces is
  \begin{itemize}
  \item an {\em $\mathcal N$-isomorphism} if the induced map
    $f(\omega)\colon X(\omega)\to Y(\omega)$ on colimits over $\cN$ is
    an isomorphism,
  \item
    an {\em $\mathcal N$-equivalence} if  $f(\omega)$
    is a weak equivalence of underlying simplicial sets, and
  \item 
    an {\em $\cI$-equivalence} if the induced
    map $f_{h\cI}\colon X_{h\cI}\to Y_{h\cI}$ on Bousfield-Kan homotopy colimits \cite[Chapter\,XII]{BK}
    is a weak equivalence of simplicial sets.
  \end{itemize}
\end{defn}

\begin{rk}
  The above definition of $\mathcal N$-equivalences is not the same as
  the one given in \cite[\S\,2.5]{Sagave-S_group-compl}, where the
  homotopy colimit over $\cN$ is used instead of the categorical
  colimit.  However, sequential colimits of simplicial sets are fully
  homotopical for weak equivalences. (One way to see this is that
  sequential colimits of acyclic Kan fibrations are again acyclic Kan
  fibrations by a lifting argument, and that therefore cofibrant
  replacements in the projective model structure on $\sset^{\cN}$ are
  mapped to weak equivalences by $\colim_{\cN}$; compare also
  Lemma~\ref{lem:N-isos-I-equivalences}.) This implies that for every
  functor $X\colon \cN\to \sset$, the canonical morphism
  \[ \hocolim_\cN X\ \to \ \colim_\cN X\]
  is a weak equivalence, compare \cite[Chapter\,XII, 3.5]{BK}.  So the two
  definitions of $\mathcal N$-equivalences of $\cI$-spaces are
  equivalent.
\end{rk}

We recall that the homotopy colimit of an $\cI$-space $X$ is related
to the homotopy colimit of the $M$-space $X(\omega)$ by a chain of
three natural weak equivalences; this observation is due to Jeff
Smith, and recorded in \cite[Proposition\,2.2.9]{shipley-THH}.
We let $\cI_\omega$ denote the category whose objects are the sets
$\mathbf m=\{1,\dots,m\}$ for $m\geq 0$, and the set $\omega=\{1,2,3,\dots\}$.
Morphisms in $\cI_\omega$ are all injective maps.  The restriction functor
\[ \sset^{\cI_\omega}\ \to \ \sset^{\cI} \]
has a left adjoint
\[ \sset^{\cI}\ \to \ \sset^{\cI_\omega}\ , \quad X \ \longmapsto \
\tilde X\ ,\]
given by left Kan extension.  Since $\cI$ is a full subcategory of
$\cI_\omega$ and the canonical functor $\cN \to \cI/\omega$
is homotopy cofinal, we can take $\tilde X=X$ on the subcategory
$\cI$, and
\[ \tilde X(\omega)\ = \ X(\omega)\ = \ \colim_{\cN} X \ .\] The
inclusion $j\colon M\to \cI_\omega$ of categories is homotopy
cofinal. Writing $L_hX$ for the homotopy left Kan extension of $X$
along $\cI \to \cI_\omega$, this implies that the canonical morphism
\[ ((L_h X)(\omega))_{h M} \ = \ \hocolim_{M} (L_h X\circ j) \ \to \ \hocolim_{\cI_\omega}L_h X \]
is another weak equivalence of simplicial sets,
see the dual version of \cite[Chapter\,XI, 9.2]{BK}.
Because sequential colimits of simplicial sets are fully homotopical,
the natural morphism from the homotopy Kan extension of $X$ at $\omega$
to the (categorical) Kan extension at $\omega$ is a weak equivalence.
So we obtain a chain of natural weak equivalences
\begin{equation}\label{eq:XhcI-identification} X_{h\cI} \ \xla{\ \sim\ } \ (L_h X)_{h\cI_\omega} \ \xla{\ \sim \ }\
   ((L_h X)(\omega))_{h M} \ \xra{\ \sim \ }\ X(\omega)_{h M}\ .\end{equation}
 This chain of equivalences proves the second item of the following proposition.
 The first item is a restatement of definitions.

 \begin{prop}\phantomsection \label{prop:omega translates}
   \begin{enumerate}[\em (i)]
   \item A morphism $f\colon X\to Y$ of $\cI$-spaces is an $\mathcal N$-equivalence 
     if and only if the morphism $f(\omega)\colon X(\omega)\to Y(\omega)$ is a weak equivalence
     of $M$-spaces.
   \item A morphism $f\colon X\to Y$ of $\cI$-spaces is an $\cI$-equivalence if and only if
     the morphism $f(\omega)\colon X(\omega)\to Y(\omega)$ is an $M$-equivalence of $M$-spaces.\qed
   \end{enumerate}
 \end{prop}

 The category $\cI$ supports a permutative structure
 \[ - \concat - \colon \cI \times \cI \to \cI  \]
 by concatenation of finite sets. On objects, it is given by $\bld{m}\concat\bld{n} = \bld{m+n}$.
 The concatenation of two morphisms
 $\alpha\colon \mathbf m\to\mathbf k$ and $\beta\colon \mathbf n\to\mathbf l$ is given by
\[ (\alpha\concat \beta)(i)\ = \
  \begin{cases}
    \qquad \alpha(i) & \text{ for $1\leq i\leq m$, and}\\
    \beta(i-m)+k & \text{ for $m+1\leq i\leq m+n$.}
  \end{cases}\]
 The {\em box product} $X\boxtimes Y$ of two $\cI$-spaces $X$ and $Y$ is the convolution product for the
 concatenation monoidal structure, i.e., the left Kan extension of the functor
 \[ \cI\times \cI \ \to\ \sset\ , \quad (
   \mathbf m,\mathbf n)\ \longmapsto \ X(\mathbf m)\times Y(\mathbf n) \]
 along $- \concat -\colon \cI\times \cI\to \cI$.

 The box product of $\cI$-spaces comes with two `projections', i.e., morphisms of $\cI$-spaces
 \[ q^1 \colon  X\boxtimes Y \ \to \ X\text{\qquad and\qquad} q^2\colon X\boxtimes Y \ \to \ Y\ .\]
 The morphism $q^1$ corresponds, via the universal property of the box product,
 to the bimorphism given by the composite
 \[ X(\mathbf m)\times Y(\mathbf n) \ \xra{\text{project}}\
   X(\mathbf m) \ \xra{X(i^1)}\ X(\bld{m}\concat\bld{n})\ ,\]
 for $m,n\geq 0$, where $i^1\colon \mathbf m\to\bld{m}\concat\bld{n}$ is the inclusion.
 The morphism $q^2$ corresponds, via the universal property of the box product,
 to the bimorphism given by the composite
 \[ X(\mathbf m)\times Y(\mathbf n) \ \xra{\text{project}}\
   Y(\mathbf n) \ \xra{Y(i^2)}\ Y(\bld{m}\concat\bld{n})\ ,\]
 where $i^2\colon \mathbf n\to\bld{m}\concat\bld{n}$ sends $i$ to $m+i$.

 We introduced the box product of $M$-spaces in Definition \ref{def:box-product}
 as the $M$-invariant subspace of disjointly supported pairs inside the product.
 The image of the morphism of $M$-spaces
 \[ (q^1(\omega),q^2(\omega))\colon (X\boxtimes Y)(\omega)\ \to \ X(\omega)\times Y(\omega) \]
 lands in the subspace $X(\omega)\boxtimes  Y(\omega)$, and we write
 \[ \tau_{X,Y}\colon    (X\boxtimes Y)(\omega)\ \to \ X(\omega)\boxtimes  Y(\omega) \]
 for the restriction of $(q^1(\omega),q^2(\omega))$ to this image.
 The next proposition says that the transformation $\tau$ is a strong symmetric monoidal structure on the functor 
 \[ (-)(\omega)\colon \sset^{\cI}\ \to \ \sset^{M}_{\tm} \ .\]

 \begin{prop}\label{prop:strong monoidal}
  For all $\cI$-spaces $X$ and $Y$, the morphism
  $\tau_{X,Y}\colon (X\boxtimes Y)(\omega)\to  X(\omega)\boxtimes  Y(\omega)$
  is an isomorphism of $M$-spaces. 
 \end{prop}
 \begin{proof}
   All structure is sight is defined dimensionwise in the simplicial direction, so it suffices
   to prove the claim for $\cI${\em-sets} (as opposed to $\cI$-spaces).
   Both sides of the transformation $\tau$ preserve colimits in each variable,
   which reduces the claim to representable $\cI$-sets.
   The convolution product of represented functors is represented, i.e.,
   \[ \cI(\mathbf m,-)\boxtimes \cI(\mathbf n,-)\ \iso \ \cI(\bld{m}\concat\bld{n},-) \ ,  \]
   and $\cI(\mathbf m,-)(\omega)$ is isomorphic to the $M$-set $\cI_m$ discussed in Example \ref{ex:I_m}.
   So the case of represented functors is taken care of by Example \ref{eg:box of semifree}.
 \end{proof}

 The next statement in particular contains Theorem~\ref{thm:intro:boxtimes-fully-homotopical}
 from the introduction. 

 \begin{thm}\label{thm:box invariance}
  For every $\cI$-space $X$, the functor $X\boxtimes -$ preserves
  $\mathcal N$-equivalences and $\cI$-equivalences of $\cI$-spaces.
\end{thm}
\begin{proof}
  We let $f\colon Y\to Z$ be a morphism of $\cI$-spaces that is an
  $\mathcal N$-equivalence or $\cI$-equivalence, respectively.
  Then by Proposition \ref{prop:omega translates}, the morphism of tame $M$-spaces
  $f(\omega)\colon Y(\omega)\to Z(\omega)$ is a weak equivalence or $M$-equivalence, respectively.
  So the morphism 
  \[  X(\omega)\boxtimes  f(\omega)\colon X(\omega)\boxtimes  Y(\omega)\ \to\ X(\omega)\boxtimes  Z(\omega)\]
  is a weak equivalence by Theorem \ref{thm:box preserves weak M},
  or an $M$-equivalence by Theorem~\ref{thm:intro:boxtimes-M-fully-homotopical}, respectively.
  By the isomorphism of Proposition \ref{prop:strong monoidal}, this means that
  \[  (X\boxtimes f)(\omega)\colon (X\boxtimes Y)(\omega)\ \to\ (X\boxtimes Z)(\omega)\]
  is a weak equivalence or an $M$-equivalence, respectively.
  Another application of Proposition \ref{prop:omega translates} proves the claim.
\end{proof}

The next statement generalizes~\cite[Corollary 2.29]{Sagave-S_group-compl}.

\begin{cor}\label{cor:box to product}
  The morphism $(q^1_{h\cI},q^2_{h\cI}) \colon (X\boxtimes Y)_{h\cI} \to 
  X_{h\cI}\times Y_{h\cI}$ is a weak equivalence of simplicial sets
  for all $\cI$-spaces $X$ and $Y$.
\end{cor}
\begin{proof}
  Using the previous theorem, this follows from~\cite[Corollary
  2.29]{Sagave-S_group-compl} by arguing with a flat cofibrant
  replacement of $X$. Alternatively, one can use the chain of weak
  equivalences~\eqref{eq:XhcI-identification} to reduce this to a
  statement about tame $M$-spaces and then argue with Proposition
  \ref{prop:strong monoidal} and Theorem \ref{thm:invariance M} to
  directly verify the claim.
\end{proof}

\section{Flat \texorpdfstring{$\cI$}{I}-spaces and tame \texorpdfstring{$M$}{M}-spaces}\label{sec:flat-I}

In this section we show that the functor
$\colim_{\cN} = (-)(\omega) \colon \sset^{\cI} \to \sset^M$
from Construction~\ref{con:M action} is a left adjoint,
and that it identifies a certain full subcategory of {\em flat} $\cI$-sets
with the category of tame $M$-sets.
We will then use this adjunction for the homotopical analysis
of flat $\cI$-spaces and tame $M$-spaces.

\subsection*{Flat \texorpdfstring{$\cI$}{I}-sets}
The following definition singles out a particular class of `flat'
$\cI$-sets. The $\cI$-spaces previously called `flat' are precisely
the ones that are dimensionwise flat as $\cI$-sets, see Remark
\ref{rk:flat is flat}.

\begin{defn}\label{def:flat I-set}
  An $\cI$-set $X$ is {\em flat} if the following two conditions hold:
  \begin{enumerate}[(i)]
  \item for every morphism $\alpha\colon \bld{m}\to\bld{n}$ in $\cI$,
    the map $X(\alpha)\colon X(\bld{m})\to X(\bld{n})$ is injective.
  \item The functor $X$ sends pullback squares in $\cI$ to pullback squares of sets. 
  \end{enumerate}
\end{defn}

\begin{rk}\label{rk:flat is flat}
The terminology just introduced is consistent with the usage of the
adjective `flat' in the context of $\cI$-spaces -- despite the fact
that the definitions look very different at first sight.
We recall from \cite[Definition~3.9]{Sagave-S_diagram} that an
$\cI$-space $X$ is {\em flat} if for every $\bld{n}$ in $\cI$, the
latching morphism
\begin{equation}\label{eq:latching-map} \nu_n^X\colon L_{\bld{n}} X\ = \ \left(\colim_{\bld{m} \to \bld{n} \in \partial
    (\cI / \bld{n})} X(\bld{m})\right)\ \to \ X(\bld{n}) 
\end{equation}
is a monomorphism of simplicial sets;
here $\partial(\cI / \bld{n})$ denotes the full subcategory of the over-category
$(\cI / {\bld{n}})$ on the objects that are non-isomorphisms.
The flatness criterion established
in~\cite[Proposition\,3.11]{Sagave-S_diagram} precisely says that
an $\cI$-space is flat if and only if the $\cI$-set of $q$-simplices
is flat in the sense of Definition \ref{def:flat I-set} for every $q\geq 0$.

The flat $\cI$-spaces are relevant for us because they are
the cofibrant objects of a \emph{flat $\cI$-model structure}
on $\sset^{\cI}$ with weak equivalences the $\cI$-equivalences
\cite[Proposition~3.10]{Sagave-S_diagram}.
\end{rk}

The following combinatorial property amounts to the fact that
every monomorphism between flat $\cI$-spaces is automatically a flat cofibration
in the sense of \cite[Definition~3.9]{Sagave-S_diagram},
a fact that does seem to have been noticed before.

\begin{prop}\label{prop:monos are flat cofibrations}
  Let $i\colon X\to Y$ be a monomorphism between flat $\cI$-sets.
  Then for every $n\geq 0$, the induced map
  \[ \nu_n^Y\cup i(\bld{n}) \colon L_n Y\cup_{L_n X} X(\bld{n}) \ \to \ Y(\bld{n})\]
is injective.
\end{prop}
\begin{proof} The argument in the proof of~\cite[Proposition\,3.11]{Sagave-S_diagram} shows
that the latching maps  $\nu_n^X$ and $\nu_n^Y$ are injections.
  We claim that for every $n\geq 0$, the commutative square of sets
  \[ \xymatrix{
      L_n X     \ar[r]^-{\nu_n^X}\ar[d]_{L_n i} & X(\bld{n})\ar[d]^{i(\bld{n})}\\
      L_n Y\ar[r]_-{\nu_n^Y} &  Y(\bld{n})
    } \]
  is a pullback. Since $i(\bld{n})$ is also injective,
  this implies that the pushout of the punctured square injects into $Y(\bld{n})$.
  Elements in the latching set $L_n Y$ are equivalence classes of pairs $(\alpha,y)$,
  where $\alpha\colon\bld{n-1}\to\bld{n}$ is an injection and $y\in Y(\bld{n-1})$.
  For the pullback property we consider an element $x\in X(\bld{n})$ such that
  \[ \nu_n^Y[\alpha,y]\ = \ \alpha_*(y)\ = \ i(\bld{n})(x) \]
  holds in $Y(\bld{n})$.
  We let $\beta,\beta'\colon \bld{n}\to\bld{n+1}$ be the two injections that
  satisfy $\beta\circ\alpha=\beta'\circ\alpha$
  and that differ on the unique element that is not in the image of $\alpha$.
  Then
  \begin{align*}
    i(\bld{n+1})(\beta_*(x))\
    &= \ \beta_*(i(\bld{n})(x))\
      = \ \beta_*(\alpha_*(y))\\
    &= \ \beta'_*(\alpha_*(y))\
      = \ \beta'_*(i(\bld{n})(x))\ =\
      i(\bld{n+1})(\beta'_*(x))\ .
  \end{align*}
  Because $i$ is a monomorphism, we conclude that $\beta_*(x)=\beta'_*(x)$.
  The following square is a pullback because $X$ is flat:
  \[\xymatrix@C=15mm{
      X(\bld{n-1})\ar[r]^-{\alpha_*}\ar[d]_{\alpha_*} &  X(\bld{n}) \ar[d]^{\beta'_*}\\
      X(\bld{n})\ar[r]_-{\beta_*}  &       X(\bld{n+1})
    }  \]
  So there is an element $z\in X(\bld{n-1})$ such that $\alpha_*(z)=x$;
  equivalently, $x$ is the image of $[\alpha,z]$ under the latching map $\nu_n^X\colon L_n X\to X(\bld{n})$.
  Now we also get
  \[
    \nu_n^Y( (L_n i)[\alpha,z])\ = \ i(\bld{n})(\nu_n^X[\alpha,z])
    \ = \ i(\bld{n})(x) \ = \  \nu_n^Y[\alpha,y]
    \ .  \]
  Since $Y$ is flat, its latching map $\nu_n^Y$ is injective, and we conclude that
  $(L_n i)[\alpha,z]=[\alpha,y]$ in $L_n Y$. This completes the proof of
  the pullback property, and hence the proof of the proposition.
\end{proof}

\begin{con}[From $M$-sets to $\cI$-sets]\label{con:tame-M-to-I}
  We let $W$ be an $M$-set and $m\geq 0$.
  We write
  \[ W_{\bld{m}}\ = \ \{x\in W\ | \ \supp(x)\subseteq \bld{m}\} \]
  for the set of elements that are supported on $\bld{m}=\{1,\dots,m\}$.
  Given an injection $\alpha:\bld{m}\to\bld{n}$, we choose an extension
  $\tilde \alpha\in M$, i.e., such that $\tilde\alpha(i)=\alpha(i)$
  for $1\leq i\leq m$. We define
  \[ \alpha_* \colon W_{\bld{m}}\ \to\  W_{\bld{n}} \]
  by $\alpha_*(x)=\tilde\alpha x$, and this is independent
  of the choice of extension by Proposition \ref{prop:finite support} (i).
  These assignments are functorial in $\alpha$, i.e., the entire data
  defines an $\cI$-set $W_{\bullet}$.
  The inclusions $W_{\bld{m}} \subseteq W$ induce a natural morphism
  $\epsilon_W:(W_{\bullet})(\omega) \to W$ of $M$-sets. 
  
  If $X$ is an $\cI$-set, we write $X^{\flat}= X(\omega)_{\bullet}$.
  The canonical map $X(\bld{m}) \to X(\omega)$ takes values in $X(\omega)_{\bld{m}}$;
  for varying $m$, these maps provide a natural transformation
  $\eta_X:X \to X(\omega)_{\bullet}=X^{\flat}$ of $\cI$-sets.
\end{con}

\begin{prop}\phantomsection\label{prop:flat-replacement}
  \begin{enumerate}[\em (i)] 
  \item
    The morphisms $\eta_X\colon X\to X^{\flat}$
    and $\epsilon_W\colon (W_\bullet)(\omega)\to W$
    are the unit and counit of an adjunction
    \[\xymatrix{
        (-)(\omega)\colon \set^\cI\  \ar@<.4ex>[r] &\  \set^M \colon (-)_\bullet\ .\ar@<.4ex>[l]
      }\] 
  \item The adjunction counit $\epsilon_W$ is injective, and $\epsilon_W$ is surjective
    if and only if $W$ is tame.
  \item The adjunction unit $\eta_X \colon X \to X^{\flat}$ is
    an $\cN$-isomorphism.
  \item The adjunction unit
    $\eta_X\colon X \to X^{\flat}$ is an isomorphism
    if and only if $X$ is flat. 
  \item The restrictions of $(-)(\omega)$ and $(-)_\bullet$
    are an adjoint equivalence of categories
  \[\xymatrix{
      (-)(\omega)\colon \set^\cI_{\mathrm{flat}}\  \ar@<.4ex>[r] &\  \set^M_{\tm} \colon (-)_\bullet\ar@<.4ex>[l]
    }\] 
  between the full subcategories of flat $\cI$-sets and tame $M$-sets.
\end{enumerate}
\end{prop}
\begin{proof}
  Parts (i) and (ii) are straightforward, and we omit their proofs.

  (iii) The composite of $\eta_X(\omega)\colon X(\omega) \to (X(\omega)_\bullet)(\omega)$
  with the adjunction counit $\epsilon_{X(\omega)}\colon (X(\omega)_\bullet)(\omega) \to X(\omega)$
  is the identity. Since the $M$-action on $X(\omega)$ is tame,
  the counit $\epsilon_{X(\omega)}$ is an isomorphism by (ii).
  So the morphism $\eta_X(\omega)$
  is an isomorphism.
  
  (iv) We suppose first that $\eta_X$ is an isomorphism.
  Then every morphism $\alpha \colon \bld{m} \to \bld{n}$ in $\cI$
  induces an injection $\alpha_*\colon W_{\bld{m}} \to W_{\bld{n}}$
  by Proposition~\ref{prop:algebraic consequences}.
  Moreover, Proposition~\ref{prop:intersect supports} shows that $W_\bullet$
  preserves pullbacks. So the $\cI$-set $W_\bullet$ is flat.

  For the converse we suppose that $X$ is flat.
  Then the maps $X(\iota_k^m)\colon X(\bld{k})\to X(\bld{m})$
  are injective, where $\iota_k^m\colon \bld{k}\to\bld{m}$ is the inclusion.
  So the canonical map $X(\bld{k})\to\colim_{\cN} X=X(\omega)$ is injective,
  hence so is its restriction $\eta_X(\bld{k})\colon X(\bld{k})\to X(\omega)_{\bld{k}}$.

  For surjectivity we consider any element of $X(\omega)_{\bld{k}}$,
  i.e., an element of $X(\omega)$ that is supported on $\bld{k}$.
  We choose a representative $x\in X(\bld{m})$ of minimal dimension,
  i.e., with $m\geq 0$ chosen as small as possible.
  We must show that $m\leq k$.
  We argue by contradiction and suppose that $m>k$.
  We let $d\in M$ be the injection defined by
  \[ d(i)\ = \
    \begin{cases}
      i & \text{ for $i< m$, and}\\
      i+1 & \text{ for $i\geq m$.}
    \end{cases}
  \]
  Then $d\,[x]=[x]$ because $d$ is the identity on $\bld{k}$
  and $[x]$ is supported on $\bld{k}$.
  This means that the elements
  \[ X(\iota_m^{m+1})(x)\ , \ X(d|_{\bld{m}})(x) \ \in \ X(\bld{m+1}) \]
  represent the same element in the colimit $X(\omega)$.
  Since the canonical map from $X(\bld{m+1})$ to $X(\omega)$ is injective,
  we conclude that  $X(\iota_m^{m+1})(x)=X(d|_{\bld{m}})(x)$.
  Since $X$ is flat, the following square is a pullback:
  \[\xymatrix@C=15mm{
      X(\bld{m-1})\ar[r]^-{X(\iota_{m-1}^m)}\ar[d]_{X(\iota_{m-1}^m)} &  X(\bld{m}) \ar[d]^{X(d|_{\bld{m}})}\\
      X(\bld{m})\ar[r]_-{X(\iota_m^{m+1})}  &       X(\bld{m+1})
    }  \]
  So there is an element $y\in X(\bld{m-1})$ such that $X(\iota_{m-1}^m)(y)=x$.
  This contradicts the minimality of $m$, so we have shown that $m\leq k$.
  
  Part (v) is a formal consequence of the other statements:
  part (iii) implies that the restricted functor $(-)_\bullet\colon \set^M_{\tm}\to \set^\cI$
  is fully faithful, and part (iv) identifies its essential image
  as the flat $\cI$-sets.
\end{proof}

\subsection*{Homotopy theory of flat \texorpdfstring{$\cI$}{I}-spaces and tame \texorpdfstring{$M$}{M}-spaces}

In Proposition \ref{prop:flat-replacement}
we have exhibited flat $\cI$-sets as a reflective subcategory
inside all $\cI$-sets, equivalent to the category of tame $M$-sets.
Restricting to {\em tame} $M$-spaces
and applying the relevant constructions in every simplicial degree provides
an adjoint functor pair:
\[\xymatrix{
    (-)(\omega)\colon \sset^\cI\  \ar@<.4ex>[r] &\  \sset^M_\tm \colon (-)_\bullet\ar@<.4ex>[l]
  }\] 
For an $\cI$-space $X$, we write
\[  X^{\flat}\  =\  (X(\omega))_{\bullet}  \]
for the composite endofunctor of $\sset^\cI$.
By Proposition \ref{prop:flat-replacement}, these constructions enjoy
the following properties:
\begin{cor}\phantomsection\label{cor:adjunction-properties}
\begin{enumerate}[\em (i)] 
\item The adjunction counit $\epsilon_W:(W_\bullet)(\omega)\to W$  is an isomorphism.
\item The adjunction unit $\eta_X \colon X \to X^{\flat}$ is
    an $\cN$-isomorphism.
  \item The adjunction unit
    $\eta_X\colon X \to X^{\flat}$ is an isomorphism
    if and only if $X$ is flat. 
  \item The restrictions of $(-)(\omega)$ and $(-)_\bullet$
    are equivalence of categories
  \[\xymatrix{
      (-)(\omega)\colon \sset^\cI_{\mathrm{flat}}\  \ar@<.4ex>[r] &\  \sset^M_{\tm} \colon (-)_\bullet\ar@<.4ex>[l]
    }\] 
  between the full subcategories of flat $\cI$-spaces and tame $M$-spaces.\qed
\end{enumerate}  
\end{cor}
\begin{rk}
Properties (ii) and (iii) in the corollary in particular say that
$X^\flat$ is a `flat replacement' of $X$, in the sense that
the adjunction unit $\eta_X\colon X\to X^\flat$
is an $\cN$-isomorphism, and hence an $\cI$-equivalence,
with a flat target.
This should be contrasted with  cofibrant replacement in the flat
model structure on $\cI$-spaces, which provides an $\cI$-equivalence
(even a level equivalence) {\em from}  a flat $\cI$-space {\em to} $X$.
\end{rk}

\begin{rk}\label{rem:alternative-pf-boxtimes-fully-homotopical}
  Corollary~\ref{cor:adjunction-properties} (ii) and the strong monoidality of
  $(-)(\omega)$ established in Proposition~\ref{prop:strong monoidal}
  also lead to an alternative proof of
  Theorem~\ref{thm:intro:boxtimes-fully-homotopical}: they imply that
  for every pair of $\cI$-spaces $X,Y$, there is a natural
  $\cN$-isomorphism $X \boxtimes Y \to X^{\flat} \boxtimes Y$ with
  $X^{\flat}$ a flat. The homotopy invariance of $\boxtimes$ then
  follows from~\cite[Proposition\,8.2]{Sagave-S_diagram} which implies
  that $X^{\flat}\boxtimes -$ preserves $\cI$-equivalences.
\end{rk}

A relatively formal consequence of this setup is that the absolute and
positive flat $\cI$-model structures on $\sset^{\cI}$
\cite[Proposition\,3.10]{Sagave-S_diagram} restrict to model
structures on the full subcategory $\sset^{\cI}_{\mathrm{flat}}$ of
flat $\cI$-spaces.  The {\em flat $\cI$-fibrations} are defined in
\cite[Definition\,3.9]{Sagave-S_diagram} as the morphisms of
$\cI$-spaces with the right lifting property against the class of flat
$\cI$-cofibrations that are also $\cI$-equivalences. In \cite[Section
6.11]{Sagave-S_diagram}, the flat $\cI$-fibrations are identified in
more explicit terms.

\begin{theorem}\phantomsection\label{thm:flat-I-on-flat}
  \begin{enumerate}[\em (i)]
  \item The classes of $\cI$-equivalences, monomorphisms and flat
    $\cI$-fibrations form the \emph{flat $\cI$-model structure} on the
    category $\sset^{\cI}_{\mathrm{flat}}$ of flat $\cI$-spaces.
  \item The classes of $\cI$-equivalences, monomorphisms that are also
    isomorphisms at~$\bld{0}$, and positive flat $\cI$-fibrations form
    the \emph{positive flat $\cI$-model structure} on the category
    $\sset^{\cI}_{\mathrm{flat}}$ of flat $\cI$-spaces.
  \item The flat and positive flat $\cI$-model structures
    on $\sset^{\cI}$ are cofibrantly generated, proper and simplicial.
  \item
    The inclusion $\sset^{\cI}_{\mathrm{flat}}\to \sset^{\cI}$
    is the right Quillen functor in a Quillen equivalence between the flat $\cI$-model structures.
\end{enumerate}
\end{theorem} 
\begin{proof}
  (i)
  Being a reflective subcategory of $\sset^{\cI}$, the category $\sset^{\cI}_{\mathrm{flat}}$
  is complete with limits created by the inclusion
  and colimits created by applying $(-)^{\flat}$
  to the corresponding colimits in $\sset^{\cI}$ (see e.g.\cite[Proposition 4.5.15]{riehl_context}). 
  Factorizations and the 2-out-of-3, retract and lifting properties are inherited from  $\sset^{\cI}$.
  The same holds for the simplicial structure, properness and the generating (acyclic) cofibrations
  claimed in part (iii).
  Proposition \ref{prop:monos are flat cofibrations} shows that every monomorphism
  between flat $\cI$-spaces is already a flat cofibration.
  So the restriction of flat cofibrations to $\sset^\cI_{\mathrm{flat}}$
  indeed yields the class of monomorphisms.
  The proof of part (ii) is very similar to that of part (i), and we omit the details. 
   
  (iv) The Quillen equivalence statement is a direct consequence
  of the fact that the adjunction unit $X\to X^\flat$ is always an $\cN$-isomorphism,
  and hence an $\cI$-equivalence, by Proposition~\ref{prop:flat-replacement} (i).  
\end{proof}

Analogously, the absolute and positive projective $\cI$-model structures 
and the level model structures studied in~\cite[\S~3]{Sagave-S_diagram}
carry over to  $\sset^{\cI}_{\mathrm{flat}}$.

\begin{cor}\phantomsection\label{cor:model-on-tame-M-spaces}
  \begin{enumerate}[\em (i)]
  \item 
    The classes of $M$-equivalences and monomorphisms 
    are part of a cofibrantly generated, proper and simplicial model structure
    on the category $\sset^{M}_{\mathrm{tame}}$ of tame $M$-spaces.
  \item 
    The classes of $M$-equivalences and monomorphisms that are also isomorphisms on $M$-fixed points
    are part of a cofibrantly generated, proper and simplicial model structure
    on the category $\sset^{M}_{\mathrm{tame}}$ of tame $M$-spaces.    
\end{enumerate}
\end{cor}
\begin{proof}
  For parts (i) and (ii) we transport the two model structures of
  Theorem~\ref{thm:flat-I-on-flat} along the equivalence of categories
  between flat $\cI$-spaces and tame $M$-spaces
  provided by Proposition \ref{prop:flat-replacement} (v).
  We note that
  Proposition~\ref{prop:omega translates} (ii) identifies the weak equivalences
  as the $M$-equivalences.
\end{proof}

\begin{rk}
  The results we proved about the homotopy theory of tame $M$-spaces
  imply that the two model structures of Corollary \ref{cor:model-on-tame-M-spaces}
  are Quillen equivalent to the Kan--Quillen model structure on the category of simplicial sets.
  For easier reference, we spell out an explicit chain of two Quillen equivalences.
  Proposition 6.23 of \cite{Sagave-S_diagram}
   shows that the colimit functor $\colim_\cI \colon \sset^\cI\to \sset$
   is a left Quillen equivalence for the projective model $\cI$-model structure
   on the category of $\cI$-spaces of \cite[Propositions\,3.2]{Sagave-S_diagram}.  
   The projective and flat $\cI$-model structures on $\sset^\cI$
   have the same weak equivalences and nested classes of cofibrations, so they are Quillen equivalent.
   The flat $\cI$-model structures on $\sset^\cI$ and on its full subcategory
   $\sset^\cI_{\text{flat}}$ are Quillen equivalent by Theorem \ref{thm:flat-I-on-flat} (iv).
   And the model structure on $\sset^M_{\tm}$ matches the model structure on 
   $\sset^\cI_{\text{flat}}$, by design.
   If we combine all this, we arrive at a chain of two Quillen equivalences,
   with left adjoints depicted on top:
   \[\xymatrix@C=15mm{
       \sset\  \ar@<-.4ex>[r]_-{\text{const}}
       &\ \sset^{\cI,\text{proj}}\  \ar@<-.4ex>[l]_-{\colim_\cI}
      \ar@<.4ex>[r]^-{(-)(\omega)}
      &\  \sset^M_{\tm} \ar@<.4ex>[l]^-{(-)_\bullet}
    }\] 
  The middle term has the projective $\cI$-model structure, and simplicial sets carry
  the Kan--Quillen model structure.
\end{rk}

The following theorem states that the positive model structure from
Corollary~\ref{cor:model-on-tame-M-spaces} lifts to commutative
monoids.

\begin{theorem}\label{thm:model-str-comm-boxtimes_M}
  The category $\text{\em Com}(\sset^M_{\tm},\boxtimes)$ of commutative $\boxtimes$-monoids
  in tame $M$-spaces admits a positive model structure
  with weak equivalences the $M$-equivalences.
  It is Quillen equivalent to the category of $E_{\infty}$ spaces. 
\end{theorem}
Since the functor
$(-)(\omega)\colon \sset^{\cI}_{\mathrm{flat}} \to \sset^M_{\tm}$ is a
strong symmetric monoidal equivalence of categories, it induces an
equivalence on the categories of commutative monoids. Thus the theorem
also gives a \emph{positive flat $\cI$-model structure} on the
category of commutative monoids in $(\sset^\cI,\boxtimes)$ whose
underlying $\cI$-spaces are flat.
\begin{proof}[Proof of Theorem~\ref{thm:model-str-comm-boxtimes_M}]
  The category $\text{Com}(\sset^M_{\tm},\boxtimes)$ is complete because
  the underlying category $\sset^M_{\tm}$ is,
  and limits in commutative monoids are created in the underlying category.
  Using~\cite[Proposition 2.3.5]{Rezk_operads}, cocompleteness follows because $\sset^M_{\tm}$
  is cocomplete by Lemma~\ref{lem:pushouts-seq-colimits-M-sets}
  (or the argument in the proof of Corollary~\ref{cor:model-on-tame-M-spaces})
  and $\boxtimes$ preserves colimits in each variable by Corollary~\ref{cor:boxtimes-preserves-colimits}.

  The model structure now arises by restricting the flat $\cI$-model
  structure on commutative $\cI$-space monoids from~\cite[Proposition
  3.15(i)]{Sagave-S_diagram} to the category of underlying flat
  commutative $\cI$-space monoids and transporting it along the
  equivalence of categories to $\text{Com}(\sset^M_{\tm},\boxtimes)$. The
  only non-obvious part is to get the factorizations. For this, it is
  sufficient to check that positive cofibrations in commutative $\cI$-space
  monoids with underlying flat domain are absolute flat cofibrations of
  $\cI$-spaces. This is a slightly stronger statement than
  \cite[Proposition 12.5]{Sagave-S_diagram} and follows from the
  argument in the proof of~\cite[Lemma 12.17]{Sagave-S_diagram}.

  For the Quillen equivalence statement, we first note
  that since $(-)(\omega)$ is strong symmetric monoidal by Proposition~\ref{prop:strong monoidal},
  its right adjoint $(-)_{\bullet}$ is lax symmetric monoidal,
  and so the composite $(-)^{\flat}$ is also lax symmetric monoidal.
  Hence $(-)^{\flat}$ also induces a left adjoint
  $\text{Com}(\sset^{\cI}) \to \text{Com} (\sset^{\cI}_{\mathrm{flat}})$ with right adjoint the inclusion.
  This adjunction is a Quillen equivalence with respect to the positive flat model structure
  since the underlying $\cI$-spaces of cofibrant commutative $\cI$-space monoids are flat.
  The claim follows since  $\text{Com}(\sset^{\cI})$
  is Quillen equivalent to $E_{\infty}$ spaces
  by~\cite[Theorem 3.6 and Proposition 9.8(ii)]{Sagave-S_diagram}. 
\end{proof}

\section{Presentably symmetric monoidal \texorpdfstring{$\infty$}{infinity}-categories}\label{sec:presentably-symmetric-monoidal}

The aim of this section is to prove Theorem~\ref{thm:presentably-sym-monoidal-intro} from the introduction.
The strategy of proof is to generalize the alternative approach
to Theorem~\ref{thm:intro:boxtimes-fully-homotopical}
outlined in Remark~\ref{rem:alternative-pf-boxtimes-fully-homotopical}.

We let $N$ be a commutative $\cI$-space monoid.
We write $\sset^{\cI}/N$ for the category of $\cI$-spaces
augmented over $N$. This over category $\sset^{\cI}/N$
inherits a symmetric monoidal convolution product given by 
\[(X \to N) \boxtimes (Y \to N) = (X \boxtimes Y \to N \boxtimes N \to N)\]
where the last map is the multiplication of $N$.
Since $(-)(\omega) \colon \sset^{\cI} \to \sset^M_{\tm}$
is strong symmetric monoidal by Proposition~\ref{prop:strong monoidal},
it induces a strong symmetric monoidal functor $\sset^{\cI}/N \to\sset^M_{\tm}/N(\omega)$.

\begin{cor}\label{cor:boxtimes-preserves-N-isos}
For every object $X\to N$ in $\sset^{\cI}/N$, the endofunctor $(X\to N)\boxtimes -$
of $\sset^{\cI}/N$ preserves $\cN$-isomorphisms. 
\end{cor}
\begin{proof}
  The map obtained by applying $(-)(\omega)$ to the product of $X \to N$
  with an $\cN$-isomorphism $Y \to Y'$ over $N$
  is isomorphic to $X(\omega) \boxtimes (Y(\omega)\to Y'(\omega))$. 
\end{proof}

\subsection*{Contravariant model structures}

If $S$ is a simplicial set, the over category $\sset/S$
admits a \emph{contravariant model structure} \cite[\S 2.1.4]{Lurie_HTT}.
It is characterized by the property that its cofibrations are the monomorphisms
and its fibrant objects are the morphisms $K \to S$
with the right lifting property against $\{ \Lambda^{n}_i \subseteq \Delta^n\, | \, 0 < i \leq n\}$.

We shall now consider $\cI$-diagrams in $\sset/S$,
and again call a morphism in  $(\sset/S)^{\cI}$ an $\cN$-isomorphism
if it induces an isomorphism when passing to the colimit of the underlying $\cN$-diagram.
Moreover, we say that a map is a \emph{contravariant $\cI$-equivalence}
if the homotopy colimit over $\cI$ formed with respect the contravariant model structure
sends it to a contravariant weak equivalence in $\sset/S$.
Since the covariant model structure is simplicial by the dual of \cite[Proposition\,2.1.4.8]{Lurie_HTT},
one may model this homotopy colimit by implementing the Bousfield--Kan formula.

\begin{lemma}\label{lem:N-isos-I-equivalences-constant-base}
  The $\cN$-isomorphisms in $(\sset/S)^{\cI}$ are contravariant
  $\cI$-equivalences. 
\end{lemma}
\begin{proof} Implementing the first two equivalences
  in~\eqref{eq:XhcI-identification} for the contravariant model
  structure shows that $\hocolim_{\cN}$-equivalences are
  $\cI$-equivalences. Thus it is sufficient to verify that the
  canonical map $\hocolim_{\cN} X \to \colim_{\cN}X$ is a
  contravariant weak equivalence. For this it is in turn sufficient to
  show that if $f\colon X \to Y$ is a map of $\cN$-diagrams in
  $\sset/S$ with each $f(\mathbf m)$ a contravariant weak equivalence,
  then $\colim_{\cN}f$ is a contravariant weak equivalence. To see
  this, we factor $f$ in the projective level model structure induced
  by the contravariant model structure as an acyclic cofibration $h$
  followed by an acyclic fibration $g$. Since $\colim_{\cN}$ is left
  Quillen with respect to the projective level model structure,
  $\colim_{\cN}h$ is a contravariant acyclic cofibration. Since the
  contravariant cofibrations coincide with the cofibrations of the
  over category model structure induced by the Kan model structure on
  $\sset$, the acyclic fibrations in the contravariant model structure
  are the maps that are acyclic Kan fibrations when forgetting the
  projection to $S$. Since acyclic Kan fibrations are characterized by
  having the right lifting property with respect to the set
  $\{ \partial \Delta^n \subseteq \Delta^n\, | \, n \geq 0\}$ and
  $\partial \Delta^n$ and $\Delta^n$ are finite, it follows that
  $\colim_{\cN}g$ is an acyclic Kan fibration when forgetting the
  projection to~$S$. This shows that $\colim_{\cN}f$ is a
  contravariant weak equivalence because it is the composite of a
  contravariant acyclic cofibration $\colim_{\cN}h$ and an acyclic Kan
  fibration $\colim_{\cN}g$.
\end{proof}

Now we let $Z$ be an $\cI$-diagram of simplicial sets,
and we consider the over category $\sset^{\cI}/Z$.
This category admits a \emph{positive contravariant $\cI$-model structure}
introduced in~\cite[Proposition\,3.10]{Nikolaus-S_presentably}.
It is defined as a left Bousfield localization of a positive contravariant level model structure,
where maps $X \to Y$ in $\sset^{\cI}/Z$ are weak equivalences or fibrations
if the maps $X(\bld{n}) \to Y(\bld{n})$ are  weak equivalences or fibrations
in the contravariant model structure on $\sset/Z(\bld{n})$
for all $\bld{n}$ in $\cI$ with $n\geq 1$.
The reason for considering the positive contravariant $\cI$-model structure
is that it can be used to give a symmetric monoidal model
for the contravariant model structure on simplicial sets
over a symmetric monoidal $\infty$-category, see~\cite[Theorem\,3.15]{Nikolaus-S_presentably}.

The positive contravariant $\cI$-model structure is somewhat difficult to work with since we are not aware of an intrinsic characterization of its weak equivalences. To identify the resulting homotopy category, we recall from~\cite[Theorem\,3.1]{Kodjabachev-S_commutative} that there is a \emph{positive Joyal $\cI$-model structure} on $\sset^{\cI}$ whose weak equivalences are the maps that induce weak equivalences in the Joyal model structure when forming the homotopy colimit with respect to the Joyal model structure. The key property of the positive contravariant $\cI$-model structure is now that when $Z$ is fibrant in the
positive Joyal $\cI$-model structure, then a zig-zag of Joyal $\cI$-equivalences between positive fibrant objects relating $Z$ to a constant $\cI$-diagram on a simplicial set $S$ induces a zig-zag of Quillen equivalences between $\sset^{\cI}/Z$ with the  positive contravariant $\cI$-model structure and $\sset/S$ with the contravariant model structure~\cite[Lemma\,3.11 and Corollary 3.14]{Nikolaus-S_presentably}. 

\begin{lemma}\label{lem:N-isos-I-equivalences}
  If $Z$ is an object of $\sset^{\cI}$ that is fibrant in the
  positive Joyal $\cI$-model structure, then an $\cN$-isomorphism
  in $\sset^{\cI}/Z$ is also a weak equivalence in the positive
  contravariant $\cI$-model structure.
\end{lemma}
\begin{proof}
We let $Z^c \to Z$ be a cofibrant replacement of $Z$ in the positive Joyal $\cI$-model structure on $\sset^{\cI}$. Then we can assume $Z^c \to Z$ to be a positive level fibration so that $Z^c \to Z$ is an acyclic Kan fibration in positive degrees. Moreover, we write $S = \colim_{\cI}Z^c$ and consider the adjunction unit $Z^c \to \const_{\cI}S$. The latter map is a Joyal $\cI$-equivalence by~\cite[Corollary\,2.4]{Kodjabachev-S_commutative} and thus a positive Joyal level equivalence since both $Z^c$ and $\const_{\cI}S$ are homotopy constant  in positive degrees with respect to the Joyal model structure. 

Given an object $X \to Z$ in $\sset^{\cI}/Z$, we get a composite of 
acyclic Kan fibrations 
\[\xymatrix@-1pc{ (X \times_{Z^c} Z)^c \ar@{->>}[rr]^-{\sim} && X \times_{Z^c} Z \ar@{->>}[rr]^-{\sim} && X}\]
where the first map is a positive $\cI$-cofibrant replacement in
$\sset^{\cI}/Z^c$ and the second map is the base change of $Z^c \to Z$
along $X \to Z$. It follows that $X$ is weakly equivalent to image of
$ (X \times_{Z^c} Z)^c$ under
$(Z^c \to Z)_!\colon \sset^{\cI}/Z^c \to \sset^{\cI}/Z$. Since both
$(Z^c \to Z)_!$ and $(Z^c \to \const_{\cI}S)_!$ are Quillen
equivalences~\cite[Proposition~3.10]{Nikolaus-S_presentably}, we
deduce that a map $f\colon X \to X'$ in $\sset^{\cI}/Z$ is a weak
equivalence in the positive contravariant $\cI$-model structure if and
only if the image of $ (f \times_{Z^c} Z)^c$ under
$(Z^c \to \const_{\cI}S)_!$ is.

Now assume that $f$ is an $\cN$-isomorphism. Since $\colim_{\cN}$
commutes with pullbacks and sends maps that are acyclic Kan fibrations
in positive levels to acyclic Kan fibrations, the image of
$f \times_{Z^c} Z $ under $(Z^c \to \const_{\cI}S)_!$ is an
$\cN$-isomorphism that is weakly equivalent in
$\sset^{\cI}/\const_{\cI}S$ to the image of $ (f \times_{Z^c} Z)^c $
under $(Z^c \to \const_{\cI}S)_!$. This reduces the claim to showing
that an $\cN$-isomorphism over a constant base is a positive
contravariant $\cI$-equivalence. In this case, the proof
of~\cite[Proposition\,3.13]{Nikolaus-S_presentably} shows that weak
equivalences in the positive contravariant $\cI$-model structure
coincide with the $\hocolim_{\cI}$-equivalences on $\cI$-diagrams in
$\sset/S$. Thus the claim follows from
Lemma~\ref{lem:N-isos-I-equivalences-constant-base}.
\end{proof}

\begin{prop}\label{prop:contravariant-I-hty-inv}
  Let $N$ be a commutative $\cI$-simplicial set that is fibrant
  in the positive Joyal $\cI$-model structure on $\sset^{\cI}$,
  and whose underlying $\cI$-simplicial set is flat.
  Then the $\boxtimes$-product on $\sset^{\cI}/N$ is homotopy invariant
  with respect to the positive Joyal $\cI$-model structure. 
\end{prop}
\begin{proof}
  Because the flat $\cI$-simplicial sets are reflective in $\cI$-simplicial sets
  and the underlying $\cI$-simplicial set of $N$ is flat,
  every morphism $f\colon X\to N$ factors as $f=\bar f\circ\eta_X$
  for a unique morphism  $\bar f\colon X^\flat\to N$, 
  where $\eta_X\colon X\to X^\flat$ is the adjunction unit.
  Then $\eta_X$ is a morphism in $\sset^{\cI}/N$ from $f\colon X\to N$ to $\bar f\colon X^\flat\to N$.
  
  Corollary~\ref{cor:boxtimes-preserves-N-isos} and
  Lemma~\ref{lem:N-isos-I-equivalences} show that
  $\eta_X\colon (f\colon X\to N)\to (\bar f\colon X^\flat\to N)$
  induces a weak equivalence $\eta_X \boxtimes (Y \to N)$ in the positive
  contravariant $\cI$-model structure for every object $Y\to N$ in $\sset^{\cI}/N$.
  So it is sufficient to show
  that $(X \to N) \boxtimes -$ preserves weak equivalences in the
  positive contravariant $\cI$-model structure when $X$ is flat. 

  To see this, we first consider the case where $X$ is of the form
  $X = K\times \cI(\bld{\bld{k}}, -)/H $,
  where   $K$ is a simplicial set, $\bld{k}$ is an object
  of $\cI$, and $H \subseteq \Sigma_k$ is a subgroup. Since the $\Sigma_k$-action on
  \[   ((K\times \cI(\bld{k}, -)) \boxtimes Y)(\bld{m})\ \iso\ 
  K\times \colim_{\bld{k}\concat \bld{l} \to \bld{m}} X(\bld{l}) \]
  is levelwise free by~\cite[Lemma\,5.7]{Sagave-S_diagram}, it follows
  that $X \boxtimes - $ preserves absolute levelwise contravariant
  $\cI$-equivalences. From here the proof proceeds as the one
  of~\cite[Proposition\,3.18]{Nikolaus-S_presentably} which addresses
  the claim for $X$ being absolute projective cofibrant (rather than
  flat).
\end{proof}

\begin{proof}[Proof of Theorem~\ref{thm:presentably-sym-monoidal-intro}]
  Proposition~\ref{prop:contravariant-I-hty-inv} implies that the
  monoidal product on the category $\sset^{\cI}/M^{\mathrm{rig}}$
  considered in~\cite[Theorem\,3.15]{Nikolaus-S_presentably} is
  homotopy invariant.  Since the homotopy invariance allows us to
  cofibrantly replace objects, it is preserved under the monoidal left
  Bousfield localization arising
  from~\cite[Proposition~2.2]{Nikolaus-S_presentably}.  Therefore, the
  proof of~\cite[Theorem~1.1]{Nikolaus-S_presentably} shows the claim.
\end{proof}

\begin{appendix}
\section{Algebras over the injection operad}\label{app:algebras-inj-operad}

In this appendix we identify the commutative monoids
in the symmetric monoidal category of tame $M$-sets
with the tame algebras over the injection operad,
see Theorem \ref{thm:Iset and com}.
We also show that for tame $M$-sets, the box product
introduced in Definition \ref{def:box-product} as an $M$-subset of the product
is isomorphic to the operadic product, see Proposition \ref{prop:convolution versus operadic}. 
The results in this appendix are combinatorial in nature,
and they are not needed for the homotopical analysis in the body of the paper;
however, we feel that Theorem \ref{thm:Iset and com} and 
Proposition \ref{prop:convolution versus operadic} are important to
put the box product of $M$-spaces into context.

\begin{con}[Injection operad]
  The injection monoid $M$ is the monoid of 1-ary operations in the
  {\em injection operad} $\cM$, an operad
  in the category of sets with respect to cartesian product.
  As before, we set $\mathbf n=\{1,\dots,n\}$ and $\omega=\{1,2,\dots\}$.
  For $n\geq 0$ we let $M(n)$ denote the set of injective maps
  from the set $\mathbf n\times\omega$ to the set $\omega$.
  The symmetric group $\Sigma_n$ acts on $M(n)$ by permuting the first
  coordinate in $\mathbf n\times\omega$: for a permutation $\sigma\in\Sigma_n$
  and an injection $\varphi\colon \mathbf n\times\omega\to\omega$ we set
  \[ (\varphi\sigma)(k,i) \ = \ \varphi(\sigma(k),i) \ .\]
  The collection of sets $\{M(n)\}_{n\geq 0}$ then becomes an operad $\cM$ 
  via `disjoint union and composition'. More formally, the operad structure maps
  \begin{align}\label{eq:operad structure}
    M(k)\times M(n_1)\times\cdots\times M(n_k) \
    &\to \ M(n_1+\cdots+n_k) \\
    (\varphi;\,\psi_1,\dots,\psi_k) \qquad
    &\longmapsto  \varphi\circ(\psi_1+\dots+\psi_k) \nonumber
  \end{align}
  are defined by setting
  \[ (\varphi\circ(\psi_1+\dots+\psi_k))(i,j) \ = \ 
    \varphi(m,\, \psi_m(i-(n_1+\dots+n_{m-1}),\, j)) \]
  where $m\in\{1,\dots,k\}$ is the unique number such that
  $n_1+\dots+n_{m-1}< i\leq n_1+\dots+n_{m}$.
\end{con}

As for operads in any symmetric monoidal category, 
a categorical operad has a category of algebras over it.

\begin{defn}\label{def:I-category}
An {\em $\cM$-set} is a set equipped with an algebra structure over the injection operad $\cM$.  
A {\em morphism} of $\cM$-sets is a morphism of algebras over $\cM$.
\end{defn}

Given an $\cM$-set $X$, we write the action of the $n$-ary operations as
\[    M(n)\times X^n\ \to \ X \ , \quad
(\varphi,x_1,\dots,x_n)\ \longmapsto \ \varphi_*(x_1,\dots,x_n)\ .\]

Because $M(1)=M$ is the injection monoid, every $\cM$-set has an
underlying $M$-set.  We call an $\cM$-set {\em tame} if the underlying
$M$-set is tame in the sense of Definition \ref{def:finitely supprted}.
Because $\mathbf 0$ is the empty set, the set $M(0)$
has a single element, the unique function
$\emptyset=\mathbf 0\times\omega\to \omega$.  So every $\cM$-set $X$
has a distinguished element $0$, the image of the action map
$M(0)\to X$. The associativity of the operad action implies that the
distinguished element $0$ of an $\cM$-set is supported on the empty set.

\begin{eg}\label{eg:abelian monoid}
  We let $A$ be an abelian monoid.
  Then $A$ becomes a `trivial' $\cM$-set: for $n\geq 0$
  and $\varphi\in M(n)$ we define 
  \[  \varphi_*\ \colon \ A^n \ \to \ A\ , \quad \varphi_*(a_1,\dots,a_n)\ = \ a_1+\dots+a_n \]
  by summing in the monoid $A$; in particular, $\varphi_*$ 
  only depends on~$n$, but not on~$\varphi$.
  The $\cM$-set associated to an abelian monoid has the special property that
  the monoid $M=M(1)$ acts trivially. 
  One can show that $\cM$-sets with trivial $M$-action `are' the abelian monoids.
\end{eg}

The support of elements in sets with an action of $M = M(1)$ was
introduced in Definition \ref{def:support}. Now we discuss the
behavior of support in the more highly structured $\cM$-sets, in
particular its interaction with $n$-ary operations for $n\ne 1$.  For
example, we will show that given any pair of finitely supported
elements $x,y$ and an injection $\varphi\in M(2)$, then
$\varphi_*(x,y)$ is finitely supported and
\[
  \supp(\varphi_*(x,y))\ \subseteq\
  \varphi(\{1\}\times\supp(x)  \cup  \{2\}\times\supp(y)  )\ .
\]
For this we need the following lemma about the orbits of the right
$M(1)^n$-action on the set $M(n)$. Since the monoid
$M(1)^n$ is not a group, the relation resulting from this
action is not symmetric and it is not a priori clear when two elements
of $M(n)$ are equivalent in the equivalence relation that it
generates.  The following lemma is a discrete counterpart of the
analogous property for the linear isometries operad, compare
\cite[Lemma I.8.1]{ekmm}.

\begin{lemma}\label{lemma:monoid action} 
  Let $n \geq 2$ and let $A_1, \dots, A_n$ be finite subsets
  $\omega$. Consider the equivalence relation on the set $M(n)$ of
  injections from $\mathbf n\times\omega$ to $\omega$ generated by the
  relation
  \[ \varphi\ \sim\ \varphi(f_1+\dots+f_n) \]
  for all $f_i\in M(1)$ such that $f_i$ is the identity on $A_i$.
  Then two elements of $M(n)$ are equivalent if and only if they 
  agree on the subsets $\{i\}\times A_i$ of $\mathbf n\times\omega$ for all $i=1,\dots,n$.
\end{lemma}
\begin{proof} The `only if` is clear since the value of $\varphi$ on
  $\{i\}\times A_i$ does not change when $\varphi$ is modified by a
  generating relation.  For the converse we consider
  $\varphi,\varphi'\in M(n)$ which agree on $\{i\}\times A_i$ for all
  $i=1,\dots,n$.  If we choose bijections between $\omega$ and the
  complements of the $A_i$ in $\omega$ we can reduce (by conjugation
  with the bijections) to the special case where all $A_i$ are empty.

  We prove the special case by induction over $n$, starting with $n=2$.
  We need to show that all injections in $M(2)$ are equivalent in 
  the equivalence relation generated by the right action of $M(1)^2$.
  We show that an arbitrary injection $\varphi$ is equivalent to the
  bijection $s\colon \mathbf 2\times\omega\to\omega$ given by $s(1,i)=2i-1$ and
  $s(2,i)=2i$.

  Case 1: suppose that $\varphi(\{1\}\times\omega)$ consists entirely 
  of odd numbers and $\varphi(\{2\}\times\omega)$  consists entirely 
  of even numbers. We define $\alpha,\beta\in M$ by
  $\alpha(i)=(\varphi(1,i)+1)/2$, respectively $\beta(i)=\varphi(2,i)/2$.
  Then we have $\varphi=s(\alpha+\beta)$, so $\varphi$ and 
  $s$ are equivalent.
  
  Case 2: suppose that $\varphi(\{1\}\times\omega)$ consists entirely 
  of even numbers and $\varphi(\{2\}\times\omega)$  consists entirely 
  of odd numbers. We use the same kind of argument as in case~1.
  
  Case 3: suppose that the image of $\varphi$ consists entirely of odd numbers.
  We define $\psi\in M(2)$ by $\psi(1,i)=\varphi(1,i)$,
  $\psi(2,2i)=\varphi(2,i)$ and $\psi(2,2i-1)=2i$.
  Define $d_+, d_-\in M(1)$ by $d_+(i)=2i$ and $d_-(i)=2i-1$.
  Then $\varphi=\psi(\text{id}+d_+)$, so $\varphi$ is equivalent to
  $\psi$, which in turn is equivalent to $\psi(\text{id}+d_-)$.
  But $\psi(\text{id}+d_-)$ satisfies the hypothesis of case 1,
  so altogether $\varphi$ and $s$ are equivalent.
  
  Case 4: suppose that the image of $\varphi$ consists entirely of even numbers.
  This can be reduced to case 2 by the analogous arguments as in case 3.
  
  Case 5: In the general case we exploit that $\varphi(\{1\}\times\omega)$ 
  contains infinitely many odd numbers or it contains infinitely many
  even numbers (or both). So we can choose an injection $\alpha\in M(1)$
  such that $\varphi(\{1\}\times\alpha(\omega))$ consists of numbers
  of the same parity.
  Similarly we can choose $\beta\in M(1)$ such that 
$\varphi(\{2\}\times\beta(\omega))$ consists of numbers of the same parity.
  But then $\varphi$ is equivalent to $\varphi(\alpha+\beta)$ which
  satisfies the hypothesis of one of the cases 1, 2, 3 or 4.
So any $\varphi\in M(2)$ is equivalent to the elements $s$.

  Now we perform the inductive step and suppose that $n\geq 2$.
  We let $\varphi,\psi\in M(n+1)$ be two injections. 
  We let $s\colon \mathbf 2\times\omega\to\omega$  be any bijection.
  By the inductive hypothesis, the injections
  \[ \varphi\circ(\text{id}_{\mathbf {n-1}\times\omega}+s^{-1})\ ,\ 
  \psi\circ(\text{id}_{\mathbf {n-1}\times\omega}+s^{-1})\ \in \ M(n)  \]
  are equivalent under the action of $M(1)^n$ by precomposition.
  We may assume without loss of generality that
  \[ \varphi\circ(\text{id}_{\mathbf {n-1}\times\omega}+s^{-1})\ =\ 
  \psi\circ(\text{id}_{\mathbf {n-1}\times\omega}+s^{-1})\circ(f_1+\dots+f_n) \]
  for some $f_1,\dots,f_n\in M(1)$.
  By the special case $n=2$, the injections $f_n s$ and $s$ in $M(2)$
  are equivalent under the action of $M(1)^2$.
  We may assume without loss of generality that
  $f_n s=s(\alpha+\beta)$ for some $\alpha,\beta\in M(1)$.
  Then
  \begin{align*}
    \varphi\
    &=\ \varphi\circ(\text{id}_{\mathbf {n-1}\times\omega}+s^{-1})\circ(\text{id}_{\mathbf {n-1}\times\omega}+s) \\ 
    &=\ \psi\circ(\text{id}_{\mathbf {n-1}\times\omega}+s^{-1})\circ(f_1+\dots+f_n)\circ(\text{id}_{\mathbf {n-1}\times\omega}+s)\\
    &=\ \psi\circ(f_1+\dots+f_{n-1}+(s^{-1}f_n s))\
      =\   \psi\circ(f_1+\dots+f_{n-1}+\alpha+\beta)\ .\qedhere
\end{align*}
\end{proof}

\begin{prop}\label{prop:support of varphi(a,b)}
  Let $X$ be an $\cM$-set, $x_i,\dots,x_n$ finitely supported elements of $X$,
  and $\varphi,\psi\in M(n)$ for some $n\geq 1$.
  \begin{enumerate}[\em (i)]
  \item 
    If $\varphi$ and $\psi$ agree on
    $\{i\}\times \supp(x_i)$ for all $i=1,\dots,n$,
    then  $\varphi_*(x_1,\dots,x_n)=\psi_*(x_1,\dots,x_n)$.
  \item The element $\varphi_*(x_1,\dots,x_n)$ is finitely supported and
    \[ \supp(\varphi_*(x_1,\dots,x_n))\ \subseteq \ \bigcup_{i=1,\dots,n}
    \varphi(\{i\}\times \supp(x_i)) \ .\]
  \item The subset $X_\tau$ of $X$ consisting of finitely supported elements
    is closed under the action of the injection operad, and hence a tame $\cM$-set.
  \end{enumerate}
\end{prop}
\begin{proof}
  (i) By Lemma~\ref{lemma:monoid action} we can assume without loss of generality
  that  $\psi=\varphi(f_1+\dots+f_n)$ for suitable $f_1,\dots,f_n\in M(1)$ such that 
  $f_i$ is the identity on $\supp(x_i)$. Then
  \begin{align*}
    \psi_*(x_1,\dots,x_n) \
    &= \ (\varphi(f_1+\dots+f_n))_*(x_1,\dots,x_n) \\
    &= \ \varphi_*((f_1)_*(x_1),\dots,(f_n)_*(x_n)) \ = \
      \varphi_*(x_1,\dots,x_n) \ .\end{align*}

    (ii) We let $u\in M(1)$ be an injection that fixes
    $\varphi(\{i\}\times \supp(x_i))$ elementwise for all $i=1,\dots,n$.
    Then $u\varphi$ agrees with $\varphi$ on
    $\{i\}\times \supp(x_i)$  for all $i=1,\dots,n$.
    By Lemma~\ref{lemma:monoid action} we can assume without loss of generality
    that  $u\varphi=\varphi(f_1+\dots+ f_n)$ for suitable $f_1,\dots,f_n\in M(1)$ such that 
    $f_i$ is the identity on $\supp(x_i)$. So
    \begin{align*}
      u_*(\varphi_*(x_1,\dots,x_n))\
      &= \   (u\varphi)_*(x_1,\dots,x_n)\ = \  (\varphi(f_1+\dots+f_n))_*(x_1,\dots,x_n)\\
      &= \  \varphi_*( (f_1)_*(x_1),\dots,(f_n)_*(x_n))\ = \ 
        \varphi_*(x_1,\dots,x_n)\ .
    \end{align*}
    This shows that $\varphi_*(x_1,\dots,x_n)$ is supported on the union
    of the sets $\varphi(\{i\}\times \supp(x_i))$. Part (iii) is just another way
    to concisely summarize part (ii).
  \end{proof}

\begin{con}[Box product of tame $\cM$-sets]\label{con:box product}
  We let $X$ and $Y$ be tame $\cM$-sets.
  The box product $X\boxtimes Y$ of the underlying $M$-sets was introduced
  in Definition \ref{def:box-product}.
  We claim that $X\boxtimes Y$
  has a preferred $\cM$-action which, moreover,
  makes it a coproduct in the category of tame $\cM$-sets.
  Indeed, a general fact about algebras over an operad is that
  the product $X\times Y$ has a coordinatewise $\cM$-action that makes
  it a product of $X$ and $Y$ in the category of $\cM$-sets.
  Proposition \ref{prop:support of varphi(a,b)} (ii) shows that the subset
  $X\boxtimes Y$ of $X\times Y$ is invariant under the action
  of the full injection operad; hence the product $\cM$-action on $X\times Y$ restricts
  to an $\cM$-action on $X\boxtimes Y$.

  Given two morphisms between tame $\cM$-sets
  $f\colon X\to X'$ and $g\colon Y\to Y'$, the map
  \[ f\boxtimes g\colon X\boxtimes Y\ \to \ X'\boxtimes Y' \]
  is a morphism of $\cM$-sets (and not just a morphism of $M$-sets).
  A formal consequence of the identification of tame $\cM$-sets
  with commutative monoids under the box product in Theorem \ref{thm:Iset and com}
  below is that the box product is actually a coproduct
  in the category of tame $\cM$-sets.
\end{con}

\begin{con}[Sum operation]\label{con:sum operation}
  We introduce a partially defined `sum' operation 
  on a tame $\cM$-set $X$;
  the operation is defined on pairs of elements with disjoint support,
  i.e., it is a map
  \[ + \colon X\boxtimes X \ \to \ X \ .\]
  Given two disjointly supported objects $(x,y)$ of the tame $\cM$-set
  $X$, we choose an injection $\varphi\colon \mathbf 2\times\omega\to\omega$
  such that $\varphi(1,j)=j$ for all $j\in \supp(x)$ and
  $\varphi(2,j)=j$ for all $j\in \supp(y)$.  Then we define the {\em
    sum} of $x$ and $y$ as
  \begin{equation}\label{eq:define_sum_elements}
    x+y \ = \ \varphi_*(x,y) \ .  
  \end{equation}
  Proposition \ref{prop:support of varphi(a,b)} implies that
  this is independent of the choice of $\varphi$, 
  and that the support of $x+y$ satisfies
  \begin{equation}\label{eq:supp_of_sum}
    \supp(x+y)\ \subseteq \ 
    \varphi(\{1\}\times \supp(x))\ \cup\ \varphi(\{2\}\times \supp(y)) \ = \ 
    \supp(x)\cup \supp(y)\ . \end{equation}
\end{con}

\begin{prop}\label{prop:sum commutative}
  Let $X$ be a tame $\cM$-set.
  \begin{enumerate}[\em (i)]
  \item For every pair $(x,y)$ of disjointly supported elements of $X$, the relation
    \[ (x,y) \ = \ (x,0) + (0,y) \]
    holds in the $\cM$-set $X\boxtimes X$. 
  \item The element $0$ satisfies $x+0=x=0+x$ for all $x\in X$.
  \item The sum map $+\colon X\boxtimes X\to X$ is a morphism of $\cM$-sets.
  \item The relation $x+y=y+x$ holds for all disjointly supported
    elements $x,y$ of $X$.
  \item The sum map is associative in the following sense: the relation
    \[ (x+y)+z \ = \ x+ (y+z) \]
    holds for every triple $(x,y,z)$ of elements in $X$
    whose supports are pairwise disjoint.
  \end{enumerate}
\end{prop}
\begin{proof}
  (i) We choose $\varphi\in M(2)$ as in the definition of $x+y$, i.e.,
  $\varphi^1 = \varphi(1,-)$ is the identity on $\supp(x)$ and
  $\varphi^2 = \varphi(2,-)$ is the identity on $\supp(y)$.  Since the
  distinguished object~0 has empty support, we have
  $\supp(x,0)=\supp(x)$ and $\supp(0,y)=\supp(y)$.  So $\varphi$ can
  also be used to define $(x,0)+(0,y)$. Hence
  \[ (x,0)+(0,y) =  \varphi_*((x,0),(0,y)) = 
    (\varphi_*(x,0),\varphi_*(0,y)) = 
    (\varphi^1_*(x),\varphi^2_*(y)) =  (x,y)\ .  \]
  
  (ii)
  The distinguished 0 has empty support.
  We choose an injection $\varphi\in M(2)$
  as in the definition of $x+0$, i.e., such that $\varphi^1$ is the identity on
  the support of $x$.
  Then
  \[ x+0 \ = \ \varphi_*(x,0)\ = \ \varphi^1_*(x)\ = \ x\ . \]
  The relation $0+x=x$ is proved in much the same way.

  (iii)
  We must show that the sum map commutes with the action of $n$-ary operations in $M(n)$,
  for every $n\geq 0$.
  The case $n=0$ is the relation $0+0=0$, which holds by (ii).
  For $n\geq 1$ we consider any $\lambda\in M(n)$, as well as
  elements $(x_j,y_j)$ in $X\boxtimes X$ for $j=1,\dots, n$.
  We choose $\varphi_j\in M(2)$ as in the definition of $x_j+y_j$, i.e.,
  $\varphi_j^1$ is the identity on $\supp(x_j)$
  and $\varphi_j^2$ is the identity on $\supp(y_j)$.
  We choose $\kappa\in M(2)$ as in the definition of $\lambda_*(x_1,\dots,x_n)+\lambda_*(y_1,\dots,y_n)$,
  i.e., $\kappa^1$ is the identity on 
  \[ \supp(\lambda_*(x_1,\dots,x_n))\ \subseteq \ \lambda\left(\bigcup_{j=1}^n \{j\}\times\supp(x_j) \right)\ ,  \]
  and $\kappa^2$ is the identity on $\supp(\lambda_*(y_1,\dots,y_n))$.
  We let $\chi\colon \mathbf{2 n}\to\mathbf{n 2}$ be the shuffle permutation defined by $\chi(1,\dots, 2n) = (1,n+1, 2, n+2, \dots, 2n)$. 
  Then the injections 
  \[ \kappa(\lambda+\lambda) \text{\quad and\quad}
    \lambda(\varphi_1+\dots+\varphi_n)(\chi\times \id_\omega)
  \]
  in $M(n+n)$ agree on the sets $\{j\}\times\supp(x_j)$ and 
  $\{n+j\}\times\supp(y_j)$ for all $j=1,\dots, n$.
  So these two injections act in the same way on the $(n+n)$-tuple $(x_1,\dots,x_n,y_1,\dots,y_n)$,
  and we deduce the relations 
  \begin{align*}
    \lambda_*(x_1,\dots,&x_n) + \lambda_*(y_1,\dots,y_n)\\
    =&\  \kappa_*(\lambda_*(x_1,\dots,x_n),\lambda_*(y_1,\dots,y_n)) \\
    =&\  (\kappa(\lambda+\lambda))_*(x_1,\dots,x_n,y_1,\dots,y_n) \\
    =&\  (\lambda(\varphi_1+\dots+\varphi_n)(\chi\times \id_\omega))_*(x_1,\dots,x_n,y_1,\dots,y_n)  \\
    =&\  (\lambda(\varphi_1+\dots+\varphi_n))_*(x_1,y_1,\dots,x_n,y_n)  \\
    =&\  \lambda_*((\varphi_1)_*(x_1,y_1),\dots,(\varphi_n)_*(x_n,y_n)) \\
    =& \ \lambda_*(x_1+y_1,\dots,x_n+y_n)\ .
  \end{align*}
  (iv)
  We showed in part (iii) that the sum map $+\colon X\boxtimes X\to X$ is a morphism of $\cM$-sets.
  Every morphism of $\cM$-sets commutes with the sum operation,
  so the sum map $+\colon X\boxtimes X\to X$ takes sums in $X\boxtimes X$ to sums in $X$.
  In other words, the interchange relation
  \begin{equation}\label{eq:interchange}
    (x+y) +(y'+z)\ = \ (x+y')+(y+z)    
  \end{equation}
  holds for all $((x,y),(y',z))$ in $(X\boxtimes X)\boxtimes (X\boxtimes X)$.
  We specialize to the case where $x=z=0$ are the distinguished elements.
  Since the distinguished element is a neutral element for the sum operation,
  we obtain the commutativity relation $y+y'=y'+y$.

  (v)
  In the special case where $y'=0$ is the neutral element,
  the interchange relation \eqref{eq:interchange}
  becomes the associativity relation $(x+y)+z =x+(y+z)$.
\end{proof}

Altogether this shows that for every tame $\cM$-set $X$,
the sum map $+\colon X\boxtimes X\to X$ is a unital, commutative and associative
morphism of $\cM$-sets. In particular, the sum map
makes the underlying $M$-set of $X$ into a commutative monoid
with respect to the symmetric monoidal structure given by the box product.
The next theorem shows that this data determines the $\cM$-action completely,
and tame $\cM$-algebras are `the same as' 
commutative monoids in the symmetric monoidal category $(\set^M_{\tm},\boxtimes)$.

\begin{theorem}\label{thm:Iset and com}
  The functor
  \[ \set^\cM_{\tm}\ \to \ \text{\em Com}(\set^M_{\tm},\boxtimes)\ , \quad X \ \longmapsto \ (X,+) \]
  is an isomorphism of categories from the category of tame $\cM$-sets
  to the category of commutative monoids
  in the symmetric monoidal category $(\set^M_{\tm},\boxtimes)$.
\end{theorem}
\begin{proof}
  The functor is clearly faithful.  Now we show that the functor is
  full.  So we let $f\colon X\to Y$ be a morphism of commutative
  $\boxtimes$-monoids in tame $M$-sets; we must show that $f$ is also
  a morphism of $\cM$-sets.  As a morphism of commutative
  $\boxtimes$-monoids, $f$ in particular preserves the distinguished
  element 0 and is compatible with the action of $M=M(1)$.

  To treat the case $n\geq 2$, we consider the map 
  \[ \chi_{X}^n\colon M(n) \times X^n \to X^{\boxtimes n}, \quad
    (\lambda, x_1, \dots, x_n) \mapsto (\lambda^1_*(x_1),\dots, \lambda^n_*(x_n))\]
  where $\lambda^i = \lambda(i,-)$. Since $\chi^n_X$ is natural in $X$,
  the upper square in the following diagram commutes:
    \begin{equation}\begin{aligned}\label{eq:composite diagram}
      \xymatrix@C=15mm{
        M(n)\times X^n \ar[r]^{M(n)\times f^n}  \ar[d]_{\chi^n_X}&
        M(n)\times Y^n  \ar[d]^{\chi^n_Y} \\
        X^{\boxtimes n} \ar[r]^{f^{\boxtimes n}} \ar[d]_{\sum} &  Y^{\boxtimes n} \ar[d]^{\sum} \\
        X \ar[r]_-f &Y    }      
    \end{aligned}\end{equation}
  Since $f$ is a morphism of commutative $\boxtimes$-monoids,
  the lower square also commutes.
  We claim that the vertical composites are the operadic action maps.
  To show this we consider $(\lambda,x_1,\dots,x_n)\in M(n)\times X^n$.
  The generalization of Proposition \ref{prop:sum commutative}~(i) to several components
  shows that
  \[ (x_1,\dots,x_n)\ = \ \sum_{j=1}^n \iota_j(x_j)\ ,\]
  where $\iota_j\colon X\to X^{\boxtimes n}$ puts $x$ in the $j$-th slot and
  fills up the other coordinates with the distinguished element $0$.
  The right hand side of this equation is the sum for the $\cM$-set
  $X^{\boxtimes n}$.  We let $\varphi\in M(n)$ be an injection as in
  the definition of $\lambda^1_*(x_1) +\dots+\lambda^n_*(x_n)$, i.e.,
  such that $\varphi^j$ is the identity on $\lambda^j(\supp(x_j))$ for
  every $1\leq j\leq n$.  Then the injections
  $\varphi(\lambda^1+\dots+\lambda^n)$ and $\lambda$ agree on
  $\{j\}\times\supp(x_j)$ for all $1\leq j\leq n$ and we thus have
$\lambda_*(\iota_j(x_j)) = (\varphi(\lambda^1+\dots+\lambda^n))_*(\iota_j(x_j)) = \lambda^j_*(x_j)$.  
  As a consequence, we have
  \[ 
  \lambda_*(x_1,\dots,x_n)\ = \ \lambda_*\left(\sum_{j=1}^n
    \iota_j(x_j)\right)\ = \ \sum_{j=1}^n
  \lambda_*\left(\iota_j(x_j)\right)\ = \ \sum_{j=1}^n
  \lambda^j_*(x_j) \]
  where the second equation is the fact that the sum functor is a
  morphism of $\cM$-sets, by Proposition \ref{prop:sum commutative}
  (iii).  Since the diagram \eqref{eq:composite diagram} commutes, the
  map $f$ is compatible with the operadic action of $M(n)$ for all $n\geq 0$,
  and so the functor is full. 

  The vertical factorization of the operadic action in~\eqref{eq:composite diagram}
  shows that it is determined by the $M$-action and the
  sum. Thus the functor is injective on objects.  It remains to show
  that it is surjective on objects, i.e., that every commutative
  $\boxtimes$-monoid $(X,0,+)$ arises from an $\cM$-set through the
  sum construction.  For $n\geq 0$, we define the operadic action map
  \[  M(n)\times X^n \ \to \ X \text{\quad by\quad}
    (\lambda,x_1,\dots,x_n)\ \longmapsto \
      \lambda_*(x_1,\dots,x_n)\ = \ \sum_{j=1}^n \lambda^j_*(x_j)\ .
  \]
  This makes sense because the maps $\lambda^j$ have disjoint images
  for $j=1,\dots,n$,
  so the elements $\lambda^1_*(x_1),\dots,\lambda^n_*(x_1)$ can indeed be added.
  For $n=0$, this definition returns the distinguished element 0, and for $n=1$
  it specializes to the given $M$-action.
  The operadic symmetry condition holds because the sum operation is commutative.
  The operadic associativity condition holds because
  the sum operation is associative and commutative.
  Finally, the sum map derived from this operadic action is the sum map
  we started out with, by definition.
\end{proof}

Now we recall the {\em operadic product}, another binary pairing for $M$-sets.
We will show that the operadic product supports a natural
map to the box product and that map is an isomorphism of $M$-sets
whenever the factors are tame.

\begin{con}[Operadic product]
  As before we denote by $M(2)$ the set of binary operations in the injection operad $\cM$,
  i.e., the set of injections from $\{1,2\}\times\omega$ to $\omega$.
  As part of the operad structure of $\cM$,
  the set $M(2)$ comes with a left $M$-action and a right
  $M^2$-action given by
  \[ M\times M(2) \times M^2\ \to \ M(2)\ , \quad
  (f,\psi,(u,v)) \ \longmapsto \ f\circ\psi\circ(u+v) \ . \]
  Here $u+v\colon \{1,2\}\times\omega\to \{1,2\}\times\omega$ is defined by
  $(u+v)(1,i)=(1,u(i))$ and $(u+v)(2,i)=(2,v(i))$.  Given two
  $M$-sets $X$ and $Y$ we can coequalize the right
  $M^2$-action on $M(2)$ with the left $M^2$-action
  on the product~$X\times Y$ and form
  \[ \ M(2) \times_{M\times M} (X\times Y)\ . \]
  The left $M$-action on $M(2)$ by postcomposition descends to an
  $M$-action on this operadic product.  Some care has to be taken when
  analyzing this construction: because the monoid $M$ is not a group,
  it may be hard to figure out when two elements of
  $M(2) \times X\times Y$ become equal in the coequalizer.  Viewed as
  a binary product on $M$-sets, $(X,Y) \mapsto M(2) \times_{M\times M} (X\times Y)$ is
  coherently associative and commutative, but it does {\em not} have a
  unit object.
\end{con}

Now we let $X$ and $Y$ be tame $M$-sets. To state the next result, we write 
\[ p^1 \colon M(2) \times_{M\times M} (X\times Y) \ \to \ X\text{\quad and\quad}
  p^2\colon M(2) \times_{M\times M} (X\times Y) \ \to \ Y\]
for the morphisms of $M$-sets defined by
\[ p^1[\psi,x,y] \ = \ \psi^1 x\text{\qquad and\qquad} p^2[\psi,x,y]\ = \ \psi^2 y\ ,\]
where $\psi^1=\psi(1,-)$ and $\psi^2=\psi(2,-)$.
We observe that
\[ \supp( p^1[\psi,x,y]) \ = \ \supp( \psi^1 x) \ \subseteq \ \psi^1(\supp(x))\
  = \ \psi(\{1\}\times\supp(x)  ) \ ,\]
and similarly $\supp( p^2[\psi,x,y])\subseteq \psi(\{2\}\times\supp(y))$.
Because $\psi$ is injective, these two sets are disjoint.
So all elements in the image of $(p^1,p^2)$ lie in $X\boxtimes  Y$.
We write
\begin{equation}\label{eq:chi}
  \chi_{X,Y}\colon M(2) \times_{M\times M} (X\times Y)\ \to\  X\boxtimes Y
\end{equation}
for the map $(p^1,p^2)$ when we restrict the codomain to $X\boxtimes Y$.

\begin{prop}\label{prop:convolution versus operadic}
  For all tame $M$-sets $X$ and $Y$, the map~\eqref{eq:chi}
 is an isomorphism of $M$-sets.
\end{prop}
\begin{proof}
  Source and target of the morphism $\chi_{X,Y}$ commute with disjoint unions
  and orbits by group actions in each of the variables.
  So Theorem \ref{thm:complement} reduces the claim to
  the special case $X=\cI_m$ and $Y=\cI_n$ for some $m,n\geq 0$.
  In this special case, the morphism $\chi_{\cI_m,\cI_n}$ factors as the following composite:
  \[ M(2) \times_{M\times M} (\cI_m\times \cI_n)\ \xra{\ q\ }\
    M(2) / (M^{(m)}\times M^{(n)})\ \xra{\ \epsilon\ }\
    \cI_{m+n}\ \xra[\iso]{\ \rho\ } \ \cI_m\boxtimes \cI_m \ .\]
  Here $M^{(m)}$ is the submonoid of $M$ consisting of those injections that
  are the identity on $\{1,\dots,m\}$.
  The first map $q$ sends the class $[\psi,f,g]$ to the class of $\psi(f+g)$;
  it is an isomorphism because the map
  \[ M/M^{(m)} \ \to \ \cI_m \ , \quad [f]\ \longmapsto \ f|_{\{1,\dots,m\}} \]
  is an isomorphism of $M$-sets.
  The second map $\epsilon$ is defined by
  \begin{equation}
    \epsilon[\psi](i)\ = \
    \begin{cases}
      \  \psi(1,i)   & \text{ for $1\leq i\leq m$, and}\\
      \psi(2,i-m) & \text{ for $m+1\leq i\leq m+n$.}
    \end{cases}
  \end{equation}
  The second map $\epsilon$ is an isomorphism of $M$-sets by
  Lemma \ref{lemma:monoid action}, applied to $n=2$, $A_1=\{1,\dots,m\}$
  and $A_2=\{1,\dots,n\}$.
  The third map $\rho$ is the isomorphism discussed in Example \ref{eg:box of semifree}.
  So the map $\chi_{\cI_m,\cI_n}$ is an isomorphism.
\end{proof}

\end{appendix}


\end{document}